\documentclass[preprint,pdftex,english,11pt]{imsart}
\usepackage{amsthm,amsmath,amssymb,amsfonts,epic,multirow,ulem}

\usepackage{thmtools}

\usepackage{tgbonum}

\usepackage[numbers]{natbib}
\RequirePackage{natbib}

\usepackage{wrapfig}
\usepackage[english]{babel}
 \usepackage[T1]{fontenc} 

 \usepackage{makeidx}
 \usepackage{fancyhdr}
\usepackage{epsf}
\usepackage[dvips]{epsfig}  
\usepackage{subfig}
 \usepackage{graphicx}
\usepackage[nice]{nicefrac}
\usepackage{dsfont}
\usepackage{calrsfs} 

\usepackage{pifont}
\usepackage{bm}

\usepackage{wasysym}
\usepackage{stmaryrd}
\usepackage{aeguill}
\usepackage{mathrsfs}
\usepackage{enumerate}
\usepackage[shortlabels]{enumitem}
\usepackage{enumitem}
\usepackage[dvipsnames,table,xcdraw]{xcolor}
\usepackage{cancel}
\usepackage{ulem}
\usepackage{caption}
\usepackage{tikz}

\usepackage{hyperref}
\definecolor{linkcolour}{rgb}{0,0.2,0.6}
\hypersetup{colorlinks,breaklinks,urlcolor=linkcolour, linkcolor=linkcolour, citecolor=linkcolour}

\usepackage{geometry}

\startlocaldefs

\newtheoremstyle{mytheoremstyle} 
        {\topsep}                    
        {\topsep}                    
        {\itshape\fontfamily{ppl}\selectfont}                   
        {}                           
        {\fontfamily{ppl}\selectfont\bfseries\color{black}}                   
        {.}                          
        {.5em}                       
        {}  
\theoremstyle{mytheoremstyle}

\newtheorem{theo}{Theorem}[section]

\newtheorem{prop}[theo]{Proposition}
\newtheorem{coro}[theo]{Corollary}

\newtheorem{lemm}[theo]{Lemma}

\newtheorem{Fact}[theo]{Fact}

\makeatletter
\renewenvironment{proof}[1][\proofname]{\par
  \pushQED{\qed}%
  \fontfamily{ppl} \topsep6\p@\@plus6\p@\relax
  \trivlist
  \item[\hskip\labelsep\itshape\bfseries#1\@addpunct{.}]\ignorespaces}{%
  \popQED\endtrivlist\@endpefalse
}

\def\R{\mathbb{R}}
\def \N{\mathbb{N}}

\def \P{\mathbb{P}} 

\def \E{\mathbb{E}} 

\newcommand{ \un }{\mathds{1}}

\def\T{\mathbb{T}}

\newenvironment{merci}{\textbf{Acknowledgments.}}{ }
\endlocaldefs

\newtheorem{Ass}{Assumption}
\newtheorem{remark}{Remark}

\renewcommand{\T}{\mathbb{T}}
\newcommand{\X}{\mathbb{X}}

\def \Eb{\mathbf{E}}
\def \Pb{\mathbf{P}}

\renewcommand{\P}{\mathbb{P}}

\newtheorem{postita}{Post-it}

\renewcommand{\P}{\mathbb{P}}

\def\T{\mathbb{T}}

\makeatletter
\newcommand*\bigcdot{\mathpalette\bigcdot@{.5}}
\newcommand*\bigcdot@[2]{\mathbin{\vcenter{\hbox{\scalebox{#2}{$\m@th#1\bullet$}}}}}
\makeatother

\newcommand{\VectCoord}[2]{#1^{(#2)}}

\newcommand{\tn}{\lfloor tn\rfloor}
\newcommand{\sn}{\lfloor sn\rfloor}
\newcommand{\rn}{\lfloor rn\rfloor}

\newcommand{\CardRoots}{\Bar{B}^{(n)}}

\DeclareMathAlphabet\mathbfcal{OMS}{cmsy}{b}{n}

\begin{document}

{\fontfamily{ppl}\selectfont

\begin{frontmatter}


\title{Local times in critical generations of a random walk in random environment on trees}

\author{\fnms{Alexis} \snm{Kagan}\ead[label=e2]{alexis.kagan@auckland.ac.nz}}
\address{Department of Statistics, University of Auckland, New Zealand. \printead{e2}} \vspace{0.5cm}

\runauthor{Kagan}


\runtitle{Local times in critical generations of a random walk in random environment on trees}

\begin{abstract}
We consider a null-recurrent randomly biased walk $\X$ on a Galton-Watson tree in the (sub)-diffusive regime and we prove that properly renormalized, the local time in a critical generation converges in law towards some function of a stable continuous-state branching process. We also provide an explicit equivalent of the probability that critical generations are reached by the random walk $\X$.
\end{abstract}

 \begin{keyword}[class=AenMS]
 \kwd[MSC2020 :  ] {60K37}, {60J80}
 \end{keyword}

\begin{keyword}
\kwd{randomly biased random walks}
\kwd{branching random walks}
\kwd{Galton-Watson trees}
\kwd{stable continuous-state branching processes}
\end{keyword}

\end{frontmatter}


\section{Introduction}

\subsection{Random walk on a Galton-Watson marked tree}\label{RBRWT}
Denote by $\N=\{0,1,2,\cdots\}$ the set of non-negative integers and $\N^*=\{1,2,\cdots\}$ the set of positive integers. Given, under a probability measure $\Pb$, a $\bigcup_{k\in\N}\R^k$-valued random variable $\mathcal{P}$ ($\mathbb{R}^0$ only contains the sequence with length $0$) with $N:=\#\mathcal{P}$ denoting the cardinal of $\mathcal{P}$, we consider the following Galton-Watson marked tree $(\T,(A_x;x\in\T))$ rooted at $e$: the generation $0$ contains one marked individual $(e,A_e)=(e,0)$. For any $n\in\N^*$, assume the generation $n-1$ has been built. If it is empty, then the generation $n$ is also empty. Otherwise, for any vertex $x$ in the generation $n-1$, let $\mathcal{P}^x:=\{A_{x^{1}},\ldots,A_{x^{N(x)}}\}$ be a random variable distributed as $\mathcal{P}$ where $N(x):=\#\mathcal{P}^x$. The vertex $x$ gives progeny to $N(x)$ marked children $(x^{1},A_{x^{1}}),\ldots,(x^{N(x)},A_{x^{N(x)}})$ independently of the other vertices in generation $n-1$, thus forming the generation $n$. We assume $\Eb[N]>1$ so that $\T$ is a super-critical Galton-Watson tree with offspring $N$, that is $\Pb(\textrm{non-extinction of }\T)>0$ and we define $\Pb^*(\cdot):=\Pb(\cdot|\textrm{non-extinction of }\T)$, where $\Eb$ (resp. $\Eb^*$) denotes the expectation with respect to $\Pb$ (resp. $\Pb^*$). \\
For any vertex $x\in\T$, we denote by $|x|$ the generation of $x$ and $x^*$ stands for the parent of $x$, that is the most recent ancestor of $x$. For convenience, we add a parent $e^*$ to the root $e$. For any $x,y\in\T$, we write $x\leq y$ if $x$ is an ancestor of $y$ ($y$ is said to be a descendent of $x$) and $x<y$ if $x\leq y$ and $x\not=y$. \\
Let us now introduce the branching potential $V:x\in\T\mapsto V(x)\in\R$: let $V(e)=A_e=0$ and for any $x\in\T\setminus{\{e}\}$
$$ V(x):=\sum_{e<z\leq x}A_z=\sum_{i=1}^{|x|} A_{x_i}. $$
Under $\Pb$, $\mathcal{E}:=(\T,(V(x);x\in\T))$ is a real valued branching random walk such that $(V(x);\; |x|=1)$ is distributed as $\mathcal{P}$. We will then refer to $\mathcal{E}$ as the random environment. \\
We are now ready to introduce the main process of our study. Given a realization of the random environment $\mathcal{E}$, we define a $\T\cup\{e^*\}$-valued nearest-neighbour random walk $\X:=(X_j)_{j\in\N}$, reflected in $e^*$ whose transition probabilities are, under the quenched probabilities $\{\P^{\mathcal{E}}_z; z\in\T\cup\{e^*\}\}$: for any $x\in\T$, $\P^{\mathcal{E}}_x(X_0=x)=1$ and 
\begin{align*}
    p^{\mathcal{E}}(x,x^*)=\frac{e^{-V(x)}}{e^{-V(x)}+\sum_{i=1}^{N(x)}e^{-V(x^i)}}\;\;\textrm{ and for all } 1\leq j\leq N(x),\;\; p^{\mathcal{E}}(x,x^j)=\frac{e^{-V(x^j)}}{e^{-V(x)}+\sum_{i=1}^{N(x)}e^{-V(x^i)}}.
\end{align*}
Otherwise, $p^{\mathcal{E}}(x,u)=0$ and $p^{\mathcal{E}}(e^*,e)=1$. Let $\P^{\mathcal{E}}:=\P^{\mathcal{E}}_e$, that is the quenched probability of $\X$ starting from the root $e$ and we finally define the following annealed probabilities
$$ \P(\cdot):=\Eb[\P^{\mathcal{E}}(\cdot)]\;\;\textrm{ and }\;\;\P^*(\cdot):=\Eb^*[\P^{\mathcal{E}}(\cdot)]. $$ 
R. Lyons and R. Pemantle \cite{LyonPema} initiated the study of the randomly biased random walk $\X$. \\
When, for all $x\in\T$, $V(x)=|x|\log\lambda$ for a some constant $\lambda>0$, the walk $\X$ is known as the $\lambda$-biased random walk on $\T\cup\{e^*\}$ and was first introduced by R. Lyons (see \cite{Lyons} and \cite{Lyons2}). The $\lambda$-biased random walk is transient (on the set of non-extinction) unless the bias is strong enough: if $\lambda\geq\Eb[N]$ then, $\Pb$-almost surely, $\X$ is recurrent (positive recurrent if $\lambda>\Eb[N]$). R. Lyons, R. Pemantle and Y. Peres (see \cite{LyonsRussellPemantle1} and \cite{LyonsRussellPemantle2}), later joined by G. Ben Arous, A. Fribergh, N. Gantert, A. Hammond \cite{BA_F_G_H} and E. Aïdékon \cite{AidekonSpeed} for example, studied the transient case and paid a particular attention to the speed $v_{\lambda}:=\lim_{n\to\infty}|X_n|/n\in[0,\infty)$ of the random walk.  \\
When the bias is random, the behavior of $\X$ depends on the fluctuations of the following $\log$-Laplace transform: for any $t\geq 0$ 
\begin{align*}
    \psi(t):=\log\Eb\Big[\sum_{|x|=1}e^{-tV(x)}\Big]=\log\Eb\Big[\sum_{|x|=1}e^{-tA_x}\Big], 
\end{align*}
and we assume that $\psi$ is finite is a neighbourhood of $1$ and that $\psi'(1)$ exists. As stated by R. Lyons and R. Pemantle \cite{LyonPema}, if $\inf_{t\in[0,1]}\psi(t)$ is positive, then $\Pb^*$-almost surely, $\X$ is transient and we refer to the work of E. Aïdékon \cite{Aidekon2008} for this case. Otherwise, it is recurrent. More specifically, G. Faraud \cite{Faraud} proved that the random walk $\X$ is $\Pb$-almost surely positive recurrent either if $\inf_{t\in[0,1]}\psi(t)<0$ or if $\inf_{t\in[0,1]}\psi(t)=0$ and $\psi'(1)>0$. It is null recurrent if $\inf_{t\in[0,1]}\psi(t)=0$ and $\psi'(1)\leq 0$. We refer for instance to \cite{HuShi10a}, \cite{HuShi10b}, \cite{AndDeb2}, \cite{HuShi15b}, \cite{HuShi10b}, \cite{AndChen} and \cite{AndAKHightPotential} for further details on the case $\inf_{t\in[0,1]}\psi(t)=0$ and $\psi'(1)=0$, also known as the slow regime for the random walk $\X$ because the largest generation reached by $(X_j)_{j\leq n}$ is of order $(\log n)^3$, see in particular \cite{HuShi10a} and \cite{HuShi10b}. \\
The present paper is dedicated to the null recurrent randomly biased walk $\X$, and we put ourselves in the following case for the random environment $\mathcal{E}$: we assume 
\begin{align}\label{DiffCase}
    \inf_{t\in[0,1]}\psi(t)=\psi(1)=0\;\;\textrm{ and }\;\;\psi'(1)<0.
\end{align}
\noindent Let us then introduce
\begin{align}\label{DefKappa}
    \kappa:=\inf\{t>1;\; \psi(t)=0\},
\end{align}
and assume $\kappa\in(1,\infty)$. Define the notation $t\land s=\min(t,s)$. We require the following:
\begin{Ass}\label{Assumption1}
Assume that there exists $\delta_1>0$ such that $\psi(t)<\infty$ for all $t\in(1-\delta_1,\kappa\land 2+\delta_1)$. Moreover
\begin{align}\label{EspCarreGen1}
    \Eb\Bigg[\underset{|x|=|y|=1}{\sum_{x\not=y}}e^{-V(x)-V(y)}\Bigg]<\infty,
 \end{align}
and if $\kappa\in(1,2]$, then
\begin{align}
    \Eb\Big[\sum_{|x|=1}\max\big(-V(x),0\big)e^{-\kappa V(x)}\Big]<\infty.
\end{align}
\end{Ass}
\begin{Ass}\label{Assumption2}
    The distribution of the $\bigcup_{k\in\N}\R^k$-valued random variable $\mathcal{P}$ is non-lattice.
\end{Ass}
\noindent When Assumptions \ref{Assumption1} and \ref{Assumption2} hold, E. Aïdékon and L. de Raphélis (see Theorem 6.1 in \cite{AidRap}) for $\kappa>2$ and L. de Raphelis (see Theorem 1 in \cite{deRaph1} for $\kappa\in(1,2]$) proved that in law, under the quenched probability $\P^{\mathcal{E}}$ for the space of càdlàg functions $D([0,\infty),\R)$, the sequence of random processes $((|X_{\tn}|/n^{1-1/(\kappa\land 2)};\; t\geq 0))$ if $\kappa\not=2$ (resp. $((|X_{\tn}|(\log n)^{1/2}/n^{1/2};\; t\geq 0))$ if $\kappa=2$) converges to the continuous-time height process associated with some stable Lévy process, see section \ref{SectionReducedProcesses} for more details. In particular, the random walk $\X$ is said to be sub-diffusive when $\kappa\in(1,2]$ and diffusive when $\kappa>2$.

\subsection{The range of the random walk $\X$ and local times}\label{SectionRange}

\noindent Let $\tau^0=0$ and for any $j\geq 1$
\begin{align*}
    \tau^j:=\inf\{k>\tau^{j-1};\; X_{k-1}=e^*,\; X_k=e\},
\end{align*}
the successive return times to the oriented edge $(e^*,e)$, with the convention $\inf\varnothing=+\infty$. When Assumptions \ref{Assumption1} and \ref{Assumption2} hold, it is known (see \cite{AndDeb1}, \cite{Hu2017} and more recently Theorem 1.2 in \cite{AK23LocalTimes}) that $\Pb^*$-almost surely, in law under $\P^{\mathcal{E}}$
\begin{align}
    \Big(\frac{\tau^n}{n^{\kappa\land 2}}\Big)\;\;\textrm{ if }\kappa\not=2\;\;\textrm{ and }\;\;\Big(\frac{\log n}{n^2}\tau^n\Big)\textrm{ if }\kappa=2\;\;\textrm{ converge in law to }\;\mathfrak{C}_{\kappa}(W_{\infty})^{\kappa\land 2}\bm{\tau}_{\kappa},
\end{align}
for some explicit constant $\mathfrak{C}_{\kappa}>0$, $\bm{\tau}_{\kappa}$ stands for the first hitting time of $-1$ by some $(\kappa\land 2)$-stable Lévy process with no negative jump and $W_{\infty}$ is the limit of the additive martingale $(\sum_{|x|=k}e^{-V(x)})_{k\geq 0}$. It is known that still under the Assumptions \ref{Assumption1} and \ref{Assumption2}, $\Pb(W_{\infty}>0)>0$, see \cite{Biggins1977}, \cite{Lyons1997}, \cite{Liu1} or \cite{Alsmeyer_Iksanov} for instance. Moreover, it is claimed in \cite{Biggins1977} that $\Pb$-almost surely, the event $\{W_{\infty}>0\}$ coincides with the event of non extinction of the underlying Galton-Watson tree $\T$. In particular, $\Pb^*(W_{\infty}>0)=1$. It is also well known that under the Assumptions \ref{Assumption1} and \ref{Assumption2}, $(W_k)$ is bounded in $\mathrm{L}^r(\Pb)$ for any $r\in[1,\kappa)$ if $\kappa\in(1,2]$ and in $\mathrm{L}^2(\Pb)$ for all $\kappa>2$. 

\vspace{0.2cm}

\noindent For any $x\in\T$ and $p\geq 1$, define the edge local time $\VectCoord{N}{p}_x$ by
\begin{align*}
    \VectCoord{N}{p}_x:=\sum_{j=1}^{\tau^p}\un_{\{X_{j-1}=x^*,\; X_j=x\}},
\end{align*}
the number of times the oriented edge $(x^*,x)$ has been visited by the random walk $\X$ up to $\tau^p$. Let us also introduce the range $\VectCoord{\mathcal{R}}{p}$ of $\X$
\begin{align*}
    \VectCoord{\mathcal{R}}{p}:=\big\{x\in\T;\; \VectCoord{N}{p}_x\geq 1\big\},
\end{align*}
the sub-tree of $\T$ made up of the vertices visited by the random walk $\X$ up to time $\tau^p$.
\begin{Fact}[Lemma 3.1, \cite{AidRap}]\label{FactGWMulti}
Under $\P$, for any $p\in\N^*$, $(\VectCoord{\mathcal{R}}{p},(\VectCoord{N}{p}_x;\; x\in\VectCoord{\mathcal{R}}{p}))$ is a multi-type Galton-Watson tree with initial type equal to $p$. Moreover, we have the following characterization: for any $x\in\T$, $x\not=e$, any $k_1,\ldots,k_{N(x)}\in\N$
\begin{align*}
    &\un_{\{N^{(1)}_x\geq 1\}}\P^{\mathcal{E}}\Big(\bigcap_{i=1}^{N(x)}\{N_{x^i}^{(1)}=k_i\}\big|N_z^{(1)};\; z\leq x\Big) \\ & =\un_{\{N^{(1)}_x\geq 1\}}\frac{(N^{(1)}_x-1+\sum_{i=1}^{N(x)}k_i)!}{(N^{(1)}_x-1)!k_1!\cdots k_{N(x)}!}\times p^{\mathcal{E}}(x,x^*)^{N^{(1)}_x}\prod_{i=1}^{N(x)}p^{\mathcal{E}}(x,x^i)^{k_i},
\end{align*}
and clearly, if $N_x^{(1)}=0$, then $N_y^{(1)}=0$ for all $y\geq x$.    
\end{Fact}

\vspace{0.2cm}
 
\noindent For any $p\geq 1$ and $k\geq 0$, let $\mathcal{G}^{(p)}_{k}$ be the sigma-algebra generated by $\{x\in\T,\; |x|\leq k;\; N_x^{(p)}\}$. Define
\begin{align}\label{DefLocalTimes}
    \VectCoord{L}{p}_k:=\sum_{j=1}^{\tau^p}\un_{\{|X_j|=k\}},
\end{align}
the local time of $(|X_j|)_{j\geq 1}$ at level $k$ and time $\tau^p$. One can notice that $L^{(p)}_k=Z^{(p)}_k+Z^{(p)}_{k+1}$ where
\begin{align}\label{DefMultiMart}
    Z^{(p)}_k:=\sum_{|x|=k}N_x^{(p)}.
\end{align}
Note that thanks to Fact \ref{FactGWMulti}, we have that $$\E\Big[\sum_{|x|=k}\VectCoord{N}{p}_x\Big|\VectCoord{\mathcal{G}}{p}_{k-1}\Big]=\E\Big[\sum_{|x|=k-1}\sum_{z;z^*=x}\VectCoord{N}{p}_z\Big|\VectCoord{\mathcal{G}}{p}_{k-1}\Big]=\sum_{|x|=k-1}
\VectCoord{N}{p}_x.$$ Hence, $(\VectCoord{Z}{p}_k/p)_{k\geq 0}$ is, under the annealed probability $\P$ and for any $p\in\N^*$, a non-negative $(\VectCoord{\mathcal{G}}{p}_k)_{k\in\N}$-martingale such that $\E[\VectCoord{Z}{p}_k/p]=\VectCoord{Z}{p}_0/p=1$ for all $k\in\N$. $(\VectCoord{Z}{p}_k/p)_{k\geq 0}$ is referred to as the \textit{multi-type additive martingale} in \cite{deRaph1}.

\vspace{0.1cm}

\noindent Note that interesting results about the behavior of $\{N^{(p)}_x;\; x\in\T\}$ exist. For instance, it is proved in \cite{ChenDeRaphMax} that under technical assumptions, $(\max\{N^{(n)}_x;\; x\in\T\}/n)_{n\geq 1}$ converges in law to an explicit positive random variable. 

\section{Main results}

If $\kappa\in(1,2]$, let us introduce
\begin{align}\label{ConstantTail}
    c_{\kappa}:=\lim_{r\to\infty}r^{\kappa}\P\Big(\sum_{x\in\T}
    \un_{\{N_x^{(1)}=1,\; \min_{e<z<x}N_z^{(1)}\geq 2\}}>r\Big),
\end{align}
and if $\kappa>2$, we define
\begin{align*}
    c_0=\frac{1}{1-e^{\psi(2)}}\Eb\Big[\sum_{x\not=y;\; |x|=|y|=1}e^{-V(x)-V(y)}\Big],
\end{align*}
with the convention $\sum_{\varnothing}=0$. The proof of the existence of $c_{\kappa}\in(0,\infty)$ under the Assumptions \ref{Assumption1} and \ref{Assumption2} is one of the main purposes of the paper of L. de Raphélis (\cite{deRaph1}, Proposition 2).\\
Let $(S_j-S_{j-1})_{j\in\N^*}$ be a sequence of\textrm{ i.i.d }real valued random variables such that $S_0=0$ and for any measurable and non-negative function $f:\R\to\R$
\begin{align}\label{LoiRW}
    \Eb[f(S_1)]=\Eb\Big[\sum_{|x|=1}e^{-V(x)}f(V(x))\Big].
\end{align}
Introduce the positive constant $C_{\infty}:=\Eb[(\sum_{j\geq 0}e^{-S_j})^{-2}]$ and
\begin{align}\label{DefCkappa}
\bm{\mathrm{C}}_{\kappa}:=\left \{\begin{array}{l c l}
    c_{\kappa}|\Gamma(1-\kappa)|  & \text{if} & \kappa\in(1,2), \\
    \frac{c_2}{2}  & \text{if} & \kappa=2, \\
    \frac{c_0}{C_{\infty}}  & \text{if} & \kappa>2,
\end{array}
\right .
\end{align}
where $\Gamma(\cdot)$ stands for the Gamma function.

\vspace{0.2cm}

\noindent We are interested in the behavior of the local time up to the first return time to the root. P. Rousselin proved (Theorem 1.2 in \cite{RousselinConduc}) that for any $\kappa>1$, $\Pb^*$-almost surely, $\P^{\mathcal{E}}(\VectCoord{L}{1}_{m}>0)/\P(\VectCoord{L}{1}_{m}>0)\to W_{\infty}$ as $m\to\infty$ but was unfortunately only able to obtain the rate of convergence of $\P^{\mathcal{E}}(\VectCoord{L}{1}_{m}>0)$ and $\P(\VectCoord{L}{1}_{m}>0)$ as $m\to\infty$ when $\kappa\in(1,2]$. Our first result allows to provide an explicit equivalent of these probabilities.
\begin{theo}\label{Theorem1}
Assume that the Assumptions \ref{Assumption1} and \ref{Assumption2} hold.
\begin{enumerate}[label=(\roman*)]
    \item If $\kappa\in(1,2)$, then in $\Pb^*$-probability
    \begin{align*}
        n\P^{\mathcal{E}}\big(\VectCoord{L}{1}_{\lfloor n^{\kappa-1}\rfloor}>0\big)\underset{n\to\infty}{\longrightarrow}W_{\infty}\big((\kappa-1)C_{\infty}c_{\kappa}|\Gamma(1-\kappa)|\big)^{-1/(\kappa-1)},
    \end{align*}
    \item If $\kappa=2$, then in $\Pb^*$-probability
    \begin{align*}
        n\P^{\mathcal{E}}\big(\VectCoord{L}{1}_{\lfloor n/\log n\rfloor}>0\big)\underset{n\to\infty}{\longrightarrow}\frac{W_{\infty}}{C_{\infty}c_2},
    \end{align*}
    \item\label{Survival3} If $\kappa>2$, then in $\Pb^*$-probability
    \begin{align*}
        n\P^{\mathcal{E}}\big(\VectCoord{L}{1}_n>0\big)\underset{n\to\infty}{\longrightarrow}\frac{W_{\infty}}{c_0},
    \end{align*}
\end{enumerate}
Moreover, for any $\kappa>1$, in $\Pb^*$-probability
\begin{align*}
    \frac{\P^{\mathcal{E}}\big(\VectCoord{L}{1}_{m}>0\big)}{\P\big(\VectCoord{L}{1}_{m}>0\big)}\underset{m\to\infty}{\longrightarrow}W_{\infty}.
\end{align*}
\end{theo}
\begin{remark}\label{BoundedInLp}
    It is not difficult to show, under the Assumptions \ref{Assumption1} and \ref{Assumption2}, that $(\P^{\mathcal{E}}(\VectCoord{L}{1}_{m}>0)/\P(\VectCoord{L}{1}_{m}>0))_m$ is bounded in $\mathrm{L}^r(\Pb)$ for any $r\in[1,\kappa)$ if $\kappa\in(1,2]$ and in $\mathrm{L}^2(\Pb)$ if $\kappa>2$. For example, it is proved in Lemma 2.3 of \cite{RousselinConduc} that $\Eb[(\P^{\mathcal{E}}(\VectCoord{L}{1}_{m}>0)/\P(\VectCoord{L}{1}_{m}>0))^r]\leq\Eb[(W_m)^r]$ for any $r\geq 1$.
\end{remark}

\begin{remark}[Critical generations]\label{CriticalGenerations}
    In view of Theorem \ref{Theorem1}, we say that $(\ell_n)_{n\geq 1}$, a sequence of positive integers, is a sequence of critical generations for $(\mathcal{R}^{(n)})_{n\geq 1}$ if $\lim_{n\to\infty}\ell_n/n^{\kappa\land 2-1}\in(0,\infty)$ (resp. $\lim_{n\to\infty}(\log n)\ell_n/n\in(0,\infty)$) when $\kappa\not=2$ (resp. $\kappa=2$).
\end{remark}

\noindent We would like to state an important proposition. In that proposition, we deal with the limiting annealed law of the rescaled local time up to the first return time to the root. For any $\lambda\geq 0$, define $\phi^{(\kappa)}(\lambda)$ by
\begin{align}\label{PhiKappa}
    \phi^{(\kappa)}(\lambda):=1-\lambda\big(1+\lambda^{\kappa\land 2-1}\big)^{-1/(\kappa\land 2-1)}.
\end{align}
One can notice that if $\kappa\geq 2$, then $\phi^{(\kappa)}(\lambda)$ is the Laplace transform of an exponential random variable with mean $1$. For any $k\in\N$, let $p^{\mathcal{E}}_k:=\P^{\mathcal{E}}(Z^{(1)}_k>0)$ and $p_k:=\Eb[p^{\mathcal{E}}_k]=\P(Z^{(1)}_k>0)$. 

\vspace{0.2cm}

\noindent Before stating our proposition, let us introduce a definition. For any integer $i\geq 1$, let $(r_{i,m})_{m\geq 1}$ be a sequence of real numbers. We say that $(r_{i,m})_{m\geq 1}$ converges to $r_{i,m}$ uniformly in $i\geq 1$ if $$\lim_{m\to\infty}\sup_{i\geq 1}|r_{i,m}-r_i|\to 0.$$ We are now ready to state our proposition.

\begin{prop}\label{UnifLawMultyMart}
Assume that the Assumptions \ref{Assumption1} and \ref{Assumption2} hold. Let $i\geq 1$ be an integer, $(\gamma_{i,m})_{m\geq 1}$ and $(\rho_{i,m})_{m\geq 1}$  be two sequences of positive integers such that $0\leq\gamma^-_m\leq\gamma_{i,m}\leq\gamma^+_m$ and $0\leq\rho^-_m\leq\rho_{i,m}\leq\rho^+_m$ where $\gamma^-_m/m\to 1$, $\gamma^+_m/m\to 1$ and $\rho^+_m/m\to 0$ as $m\to\infty$. For any $\lambda_1, \lambda_2\geq 0$, we have, uniformly in $i\geq 1$ that by
\begin{align*}
    \lim_{m\to\infty}\Big(\E\Big[e^{-\lambda_1 p_{m}L^{(1)}_{\gamma_{i,m}}}e^{-\lambda_2p_mL^{(1)}_{\gamma_{i,m}+\rho_{i,m}}}\Big]\Big)^{1/p_m}=e^{-\big(1-\phi^{(\kappa)}(2\lambda_1+2\lambda_2)\big)}.
\end{align*}
\end{prop}

\noindent As a consequence of Proposition \ref{UnifLawMultyMart}, one can deduce the following result (which turns out to be an annealed version of Corollary \ref{Corollary1} that will be stated soon). Let $\theta\geq 0$ and $k\in\N$. If we denote $$\E\Big[e^{-\theta L^{(1)}_k}\big|L^{(1)}_k>0\Big]:=\frac{\E\Big[e^{-\theta L^{(1)}_k}\un_{\{L^{(1)}_k>0\}}\Big]}{\P\big(L^{(1)}_k>0\big)}, $$ the annealed Laplace transform of $L^{(1)}_k$ conditionally on the event $\{L^{(1)}_k>0\}$, which is different from $\Eb[\E^{\mathcal{E}}[e^{-\theta L^{(1)}_k}|L^{(1)}_k>0]]$, then for any $\lambda\geq 0$ 
\begin{enumerate}[label=(\roman*)]
    \item if $\kappa\in(1,2)$, then
    \begin{align*}
        \E\left[e^{-\frac{\lambda}{n}L^{(1)}_{\lfloor n^{\kappa-1}\rfloor}}\Big|L^{(1)}_{\lfloor n^{\kappa-1}\rfloor}>0\right]\underset{n\to\infty}{\longrightarrow}\phi^{(\kappa)}\Big(2\lambda\big((\kappa-1)C_{\infty}c_{\kappa}|\Gamma(1-\kappa)|\big)^{1/(\kappa-1)}\Big);
    \end{align*}
    \item if $\kappa=2$, then 
     \begin{align*}
        \E\left[e^{-\frac{\lambda}{n}L^{(1)}_{\lfloor n/\log n\rfloor}}\Big|L^{(1)}_{\lfloor n/\log n\rfloor}>0\right]\underset{n\to\infty}{\longrightarrow}\phi^{(\kappa)}\big(2\lambda C_{\infty}c_2\big);
    \end{align*}
    \item if $\kappa>2$, then
    \begin{align*}
        \E\left[e^{-\frac{\lambda}{n}L^{(1)}_n}\Big|L^{(1)}_n>0\right]\underset{n\to\infty}{\longrightarrow}\phi^{(\kappa)}\big(2\lambda c_0\big).
    \end{align*}
\end{enumerate}
Proposition \ref{UnifLawMultyMart} will be used to prove the next theorem. Before stating the results, we need to introduce a few definitions. Let $(\mathcal{Y}^{(\kappa)}_a;\;a\geq 0)$ be a continuous state branching process (CSBP) with branching mechanism $\lambda\mapsto C_{\infty}\bm{\mathrm{C}}_{\kappa}\lambda^{\kappa\land 2}$, in the sense that for almost every environment $\mathcal{E}$, $(\mathcal{Y}^{(\kappa)}_a;\; a\geq 0)$ is a real-valued Markov process such that $\mathcal{Y}^{(\kappa)}_0=1$ and for any $\lambda\geq 0$ and any $0\leq a\leq b$
\begin{align}\label{LaplaceCSBP}
    \E^{\mathcal{E}}\Big[e^{-\lambda\mathcal{Y}^{(\kappa)}_b}\big|\mathcal{Y}^{(\kappa)}_a\Big]= \exp\Big(-\lambda\VectCoord{\mathcal{Y}}{\kappa}_a\big(1+(b-a)(\kappa\land 2-1)C_{\infty}\bm{\mathrm{C}}_{\kappa}\lambda^{\kappa\land 2-1}\big)^{-1/(\kappa\land 2-1)}\Big).
\end{align} 
When $\kappa\geq 2$, the random process $(\mathcal{Y}^{(\kappa)}_a;\;a\geq 0)$ has continuous paths and is also referred to as the Feller diffusion, that is the unique strong solution of
\begin{align*}
    \mathrm{d}\mathcal{Y}^{(\kappa)}_a=\big(2C_{\infty}\bm{\mathrm{C}}_{\kappa}\mathcal{Y}^{(\kappa)}_a\big)^{1/2}\mathrm{d}B_a\;\;\textrm{ and }\;\;\mathcal{Y}^{(\kappa)}_0=1,
\end{align*}
where $(B_a;\; a\geq 0)$ is a standard Brownian motion. \\
Let us now introduce the following definition. Let $(\VectCoord{Y}{n})_{n\geq 1}$ be a sequence of random variables taking values in a metric space $\mathcal{Q}$ and let $Y$ be a $\mathcal{Q}$-valued random variable. We say that $(\VectCoord{Y}{n})_{n\geq 1}$ converges to $Y$ in $\Pb^*$-law if for any increasing sequence of positive integers $(n_q)_{q\geq 1}$, there exists an increasing sequence of positive integers $(m_{\ell})_{\ell\geq 1}$ such that $\Pb^*$-almost surely, the sequence $(Y^{(n_{m_{\ell}})})_{\ell\geq 1}$ converges in law to $Y$ under the quenched probability $\P^{\mathcal{E}}$. Note that the convergence in $\Pb^*$-law implies that for any bounded and continuous function $F:\mathcal{Q}\to\R$, in $\Pb^*$-probability and in $\mathrm{L}^r(\Pb^*)$ for any $r>0$
\begin{align*}
    \E^{\mathcal{E}}\big[F\big(Y^{(n)}\big)\big]\underset{n\to\infty}{\longrightarrow}\E^{\mathcal{E}}\big[F\big(Y\big)\big].
\end{align*}
In particular, the convergence in $\Pb^*$-law implies the convergence in law under the probability $\P^*$. \\
\noindent Our next result is dedicated to the convergence of the rescaled local time up to the $n$-th return time to the root.
\begin{theo}\label{Theorem3}
    Assume that the Assumptions \ref{Assumption1} and \ref{Assumption2} hold. Let $a>0$.
    \begin{enumerate}[label=(\roman*)]
        \item if $\kappa\in(1,2)$, then in $\Pb^*$-law
            \begin{align*}
                \frac{1}{n}L^{(n)}_{\lfloor an^{\kappa-1}\rfloor}\underset{n\to\infty}{\longrightarrow}2W_{\infty}\mathcal{Y}^{(\kappa)}_{a/(W_{\infty})^{\kappa-1}}\;; 
            \end{align*}
        \item if $\kappa=2$, then in $\Pb^*$-law
            \begin{align*}
                \frac{1}{n}L^{(n)}_{\lfloor an/\log n\rfloor}\underset{n\to\infty}{\longrightarrow}2W_{\infty}\mathcal{Y}^{(2)}_{a/W_{\infty}}\;; 
            \end{align*}
        \item if $\kappa>2$, then in $\Pb^*$-law 
        \begin{align*}
                \frac{1}{n}L^{(n)}_{\lfloor an\rfloor}\underset{n\to\infty}{\longrightarrow}2W_{\infty}\mathcal{Y}^{(\kappa)}_{a/W_{\infty}}\;. 
            \end{align*}
    \end{enumerate}
\end{theo}
\noindent Theorem \ref{Theorem3} reminds us of very well-known results about the convergence of rescaled critical Galton-Watson processes, see for instance \cite{DuqLeGall}. Actually, we prove that somehow, the total amount of time spent by the random walk $\X$ in generation $\lfloor n^{\kappa\land 2-1}\rfloor$ if $\kappa\not=2$, $\lfloor n/\log n\rfloor$ if $\kappa=2$ up to time $\tau^n$ behaves like the the number of individuals in generation $\lfloor n^{\kappa\land 2-1}\rfloor$ if $\kappa\not=2$, $\lfloor n/\log n\rfloor$ if $\kappa=2$ of a population evolving according to some critical Galton-Watson process with $2nW_{\infty}$ initial individuals, which explains the presence of a CSBP in the right-hand side in Theorem \ref{Theorem3}.

\vspace{0.2cm}

\noindent Our last result is a consequence of Theorem \ref{Theorem1} and Theorem \ref{Theorem3}. It is an annealed result for the local time conditionally on the survival up to the first return time to the root.
\begin{coro}\label{Corollary1}
    Assume that the Assumptions \ref{Assumption1} and \ref{Assumption2} hold. Let $\lambda\geq 0$
\begin{enumerate}[label=(\roman*)]
    \item If $\kappa\in(1,2)$, then in $\Pb^*$-probability
    \begin{align*}
        \E^{\mathcal{E}}\left[e^{-\frac{\lambda}{n}L^{(1)}_{\lfloor n^{\kappa-1}\rfloor}}\Big|L^{(1)}_{\lfloor n^{\kappa-1}\rfloor}>0\right]\underset{n\to\infty}{\longrightarrow}\phi^{(\kappa)}\Big(2\lambda\big((\kappa-1)C_{\infty}c_{\kappa}|\Gamma(1-\kappa)|\big)^{1/(\kappa-1)}\Big);
    \end{align*}
    \item If $\kappa=2$, then in $\Pb^*$-probability
     \begin{align*}
        \E^{\mathcal{E}}\left[e^{-\frac{\lambda}{n}L^{(1)}_{\lfloor n/\log n\rfloor}}\Big|L^{(1)}_{\lfloor n/\log n\rfloor}>0\right]\underset{n\to\infty}{\longrightarrow}\phi^{(\kappa)}\big(2\lambda C_{\infty}c_2\big);
    \end{align*}
    \item If $\kappa>2$, then in $\Pb^*$-probability
    \begin{align*}
        \E^{\mathcal{E}}\left[e^{-\frac{\lambda}{n}L^{(1)}_n}\Big|L^{(1)}_n>0\right]\underset{n\to\infty}{\longrightarrow}\phi^{(\kappa)}\big(2\lambda c_0\big).
    \end{align*}
\end{enumerate}
\end{coro}

\noindent Corollary \ref{Corollary1} says in particular that for any $\kappa>1$, the sequences $(\P^{\mathcal{E}}(L^{(1)}_{\lfloor n^{\kappa\land 2-1}\rfloor}/n\in\cdot\;|L^{(1)}_{\lfloor n^{\kappa\land 2-1}\rfloor}>0))_{n\geq 1}$ if $\kappa\not=2$ and $(\P^{\mathcal{E}}(L^{(1)}_{\lfloor n/\log n\rfloor}/n\in\cdot\;|L^{(1)}_{\lfloor n/\log n\rfloor}>0))_{n\geq 1}$ if $\kappa=2$ both converge towards the law of a random variable which is proportional to the random variable with Laplace transform given by $\phi^{(\kappa)}$. This is strongly reminiscent of results previously obtained for critical Galton-Watson processes with one initial individual, see for instance \cite{yaglom1947}, \cite{KestenNeySpitzer}, \cite{Slack1968} and \cite{Slack1972}.

\vspace{0.2cm}

\vspace{0.3cm}

\noindent The rest of the paper is organized as follows: Section \ref{SectionProofs} is dedicated to the proofs of our main results. At the beginning of Section \ref{SectionProofs}, we give a few well-known facts about $\X$ and we introduce the reduced range (see section \ref{SectionReducedProcesses}) which turns out to be a very interesting tool to prove Theorem \ref{Theorem1} (see section \ref{ProofTheorem1}) and Theorem \ref{Theorem3} (see section \ref{ProofTheorem3}). In order to prove our proposition properly, we present a few preliminary results (see section \ref{SectionPR}). After this, the proof of Proposition \ref{UnifLawMultyMart} (see section \ref{ProofsOfPropositions}) is presented. We finally end the paper with the proof of Corollary \ref{Corollary1} (see section \ref{ProofCorollary}).

\section{Proofs of the results}\label{SectionProofs}

\noindent In all this section, the Assumptions \ref{Assumption1} and \ref{Assumption2} hold. 

\vspace{0.2cm}

\noindent We first present a few well-known facts about the edge local times (see for instance Lemma 3.1 in \cite{AndAKHightPotential}). For any $x\in\T$ 
\begin{align}\label{ProbaAlpha}
    \P^{\mathcal{E}}(\VectCoord{N}{1}_x\geq 1)=\frac{e^{-V(x)}}{H_x}\;\;\textrm{ and }\;\;\P^{\mathcal{E}}_{x^*}(\VectCoord{N}{1}_x\geq 1)=1-\frac{1}{H_x},
\end{align}
where $H_x:=\sum_{e\leq w\leq x}e^{V(w)-V(x)}$. Moreover, under $\P^{\mathcal{E}}_{x^*}$, $\VectCoord{N}{1}_x$ follows a Geometric law on $\N$ with probability of success $1-\P^{\mathcal{E}}_{x^*}(\VectCoord{N}{1}_x\geq 1)$. In particular, $\E^{\mathcal{E}}[N^{(1)}_x]=e^{-V(x)}$.

\subsection{Reduced range}\label{SectionReducedProcesses}

We introduce what we call the reduced range. The reduced range and its associated processes will be particularly useful in the proofs of Theorem \ref{Theorem1} (see section \ref{ProofTheorem1}) and Theorem \ref{Theorem3} (see section \ref{ProofTheorem3}). Roughly speaking, the reduced range associated with $\mathcal{R}^{(p)}$ is a range for which we get rid of vertices that do not give a sufficient contribution. We will see that it is easier to deal with a reduced \textit{multi-type additive martingale} than the actual one. For any $\ell\in\N$ and $p\in\N^*$, define the set $\mathcal{B}^{(p)}_{\ell}$ of vertices $x\in\T$ in generation larger than $\ell$ such that the oriented edge $(x^*,x)$ is visited exactly once before $\tau^p$ and for any ancestor $z\not=x$ of $x$ in generation larger than $\ell$, the oriented edge $(z^*,z)$ is visited at least twice before $\tau^p$, that is
\begin{align}\label{SetRoots}
    \mathcal{B}^{(p)}_{\ell}:=\big\{x\in\T,\; |x|>\ell;\; N_x^{(p)}=1\textrm{ and }\min_{\ell<i<|x|}N_{x_i}^{(p)}\geq 2\big\},
\end{align}
with the convention $\mathcal{B}^{(0)}_{\ell}=\varnothing$. This set was first introduced by E. Aïdékon and L. de Raphélis in \cite{AidRap}. $B^{(p)}_{\ell}$ stands for the cardinal of $\mathcal{B}^{(p)}_{\ell}$. We know thanks to Lemma 2.6 in \cite{AK23LocalTimes} that for any $r>0$, $\Pb$-almost surely, $\lim_{n\to\infty}\P^{\mathcal{E}}\big((\mathcal{B}^{p}_{\lfloor(\log n)^2\rfloor})_{p\leq\lfloor n^{r}\rfloor}\textrm{ is non-decreasing},\; \sup_{x\in\mathcal{B}^{(\lfloor n^r\rfloor)}_{\lfloor(\log n)^2\rfloor}}|x|\leq\lfloor(\log n)^3\rfloor\big)=1$ where $(\mathcal{B}^{p}_{\lfloor(\log n)^2\rfloor})_{p\leq\lfloor n^{r}\rfloor}$ non-decreasing means here that $\mathcal{B}^{p}_{\lfloor(\log n)^2\rfloor}\subset \mathcal{B}^{p'}_{\lfloor(\log n)^2\rfloor}$ for any $1\leq p\leq p'\leq\lfloor n^r\rfloor$. Moreover, by Lemma 2.7 in \cite{AK23LocalTimes} that $\Pb^*$-almost surely
\begin{align*}
    \lim_{n\to\infty}\P^{\mathcal{E}}\Big(\forall\; |y|\geq\lfloor(\log n)^3\rfloor,\exists\; x\in\mathcal{B}^{(\lfloor n^r\rfloor)}_{\lfloor(\log n)^2\rfloor}:\; x\leq y\Big)=1.
\end{align*}
Also, thanks to Lemma 2.9 in \cite{AK23LocalTimes}, we know that for any $r>0$, in $\P$-probability (and $\P^*$-probability)
\begin{align}\label{ConvCardRoot}
    \frac{1}{n^r}B^{(\lfloor n^r\rfloor)}_{\lfloor(\log n)^2\rfloor}\underset{n\to\infty}{\longrightarrow}W_{\infty}.
\end{align}
Let $\mathcal{B}^{(n)}:=\mathcal{B}^{(n^2)}_{\lfloor(\log n)^2\rfloor}$, $\Bar{\mathcal{B}}^{(n)}:=\mathcal{B}^{(n)}_{\lfloor(\log n)^2\rfloor}$ and denote by $B^{(n)}$ the cardinal of $\mathcal{B}^{(n)}$ and $\Bar{B}^{(n)}$ the cardinal of $\Bar{\mathcal{B}}^{(n)}$. For any $i\in\{1,\ldots,B^{(n)}\}$ let $e_i^{(n)}$ (to simplify, we will use the notation $e_i$ instead) be the $i$-th element of $\mathcal{B}^{(n)}$ visited by the random walk $\X$ and define the tree $\mathbfcal{R}^{(n)}_i$ rooted at $e_i$ by
\begin{align}\label{ReducedRange}
    \mathbfcal{R}^{(n)}_i:=\{x\in\T;\; x\geq e_i\}\cap\mathcal{R}^{(n^2)},
\end{align}
where we recall that $\mathcal{R}^{(n^2)}=\{x\in\T;\; N_x^{(n^2)}\geq 1\}$. By definition, the oriented edge $((e_i)^*,e_i)$ is visited exactly once by the random walk $\X$ up to time $\tau^{n^2}$ so one should think of $\mathbfcal{R}^{(n)}_i$ as the range of a random walk on $\T$ up to its return time to the edge $((e_i)^*,e_i)$. Precisely, for any $i\in\{1,\ldots,B^{(n)}\}$ and $x\in\mathbfcal{R}^{(n)}_i$, let $$\bm{t}_x^{(n)}:=N_x^{(n^2)}.$$ Let $((\mathbfcal{R}^{(n)}_i,\; (\bm{t}_x^{(n)};\; x\in\mathbfcal{R}^{(n)}_i));\; i>B^{(n)})$, be, under $\P$, a collection of\textrm{ i.i.d }copies of $(\mathcal{R}^{(1)},\; (N_x^{(1)},\; x\in\mathcal{R}^{(1)}))$, independent of any random variable we have introduced so far. We still denote by $e_i$ the root of $\mathbfcal{R}_i^{(n)}$ when $i>B^{(n)}$. Therefore, $((\mathbfcal{R}^{(n)}_i,\; (\bm{t}_x^{(n)};\; x\in\mathbfcal{R}^{(n)}_i));\; i\geq 1)$ is a collection of\textrm{ i.i.d }copies of the multi-type Galton-Watson tree $(\mathcal{R}^{(1)},\; (N_x^{(1)},\; x\in\mathcal{R}^{(1)}))$ with initial type $1$, see Fact \ref{FactGWMulti}. \\
The idea of considering this new range is due to E. Aïdékon and L. de Raphelis in the paper \cite{AidRap} and is significantly easier to deal with than the actual range. \\
Define $\mathbfcal{F}^{(n)}:=(\mathbfcal{R}^{(n)}_i)_{i\geq 1}$, denote by $\bm{F}^{(n)}(q)$ the number of vertices in the first $q$ trees of the forest $\mathbfcal{F}^{(n)}$ and $\bm{H}^{(n)}$ stands for the height function of the forest $\mathbfcal{F}^{(n)}$ as defined below: \\
let $\mathfrak{T}$ be a tree rooted at $u(0)$ and $\#\mathfrak{T}$ stands for the number of vertices of $\mathfrak{T}$ and assume that $\#\mathfrak{T}<\infty$. For any $i\in\{0,\ldots,\#\mathfrak{T}-1\}$, denote by $u(i)$ the $(i+1)$-th vertex of the tree $\mathfrak{T}$ for the depth first search order. We then define $H^{\mathfrak{T}}:\{0,\ldots,\mathfrak{T}\}\to\N$ to be the height function of the tree $\mathfrak{T}$, with the convention $H^{\mathfrak{T}}(\#\mathfrak{T})=0$. Then, for any $i\in\{0,\ldots,\#\mathfrak{T}-1\}$, $H^{\mathfrak{T}}(i)$ is the generation of the vertex $u(i)$ in the tree $\mathfrak{T}$. We extend this definition to a collection of trees (or forest) $\mathfrak{F}=(\mathfrak{T}_j)_{j\geq 1}$: we define the height function $H^{\mathfrak{F}}:\N\to\N$ of the forest $\mathfrak{F}$ by setting: for any $i\in\N$, $H^{\mathfrak{F}}(i)=H^{\mathfrak{T}_j}(i-\sum_{l<j}\#\mathfrak{T}_l)$ if and only if $\sum_{l<j}\#\mathfrak{T}_l\leq i<\sum_{l\leq j}\#\mathfrak{T}_l$, for some $j\geq 1$, with the convention $\#\mathfrak{T}_0=0$.

\vspace{0.2cm}

\noindent Before stating an important fact, we need a few more definitions. Let $\kappa>1$ and define $Y_{\kappa}:=(Y_{{\kappa},s};\; s\geq 0)$ to be a $(\kappa\land 2)$-stable Lévy process with no negative jump and Laplace exponent $\lambda\in[0,\infty)\mapsto \bm{\mathrm{C}}_{\kappa}\lambda^{\kappa\land 2}$ (see \eqref{DefCkappa} for the definition of $\bm{\mathrm{C}}_{\kappa}$) such that, almost surely, $\inf_{s\geq 0}Y_{\kappa,s}=-\infty$ and denote by $\bm{H}_{\kappa}:=(\bm{H}_{\kappa}(s);\; s\geq 0)$ the continuous-time height process associated with $Y_{\kappa}$. For any $s\geq 0$, $\bm{H}_{\kappa}(s)$ is defined as the "local time measure" of the set $\{t\leq s;\; Y_{\kappa,t}=\inf_{r\in[t,s]}Y_{\kappa,r}\}$, see Definition 1.2.1 in \cite{DuqLeGall} for a rigorous definition. Besides, by Theorem 1.4.3 in \cite{DuqLeGall}, $\bm{H}_{\kappa}$ has continuous paths almost-surely. One can notice that for any $\kappa\geq 2$, $\bm{H}_{\kappa}$ is distributed as $(2/\bm{\mathrm{C}}_{\kappa})^{1/2}|B|$, with $B$ a standard Brownian motion. There is a strong link between $\bm{H}^{(n)}$ and $\bm{H}_{\kappa}$, see Fact \ref{JointConvForest}. \\
For any $t\geq 0$, define $\tau_{-t}(Y_{\kappa})$ to be the first hitting time of $-t$ by $Y_{\kappa}$, that is
\begin{align*}
    \tau_{-t}(Y_{\kappa}):=\inf\{s>0;\; Y_{\kappa,s}=-t\},
\end{align*}
which is finite almost surely. Let $\bm{c}_{\infty}:=\Eb[(\sum_{j\geq 1}e^{-S_j})^{-1}]$, see \eqref{LoiRW} for the definition of $(S_j)$.
\begin{Fact}\label{JointConvForest}
For any $\kappa>1$
\begin{enumerate}[label=(\roman*)]
    \item\label{JointConvForest1} if $\kappa\not=2$, then in $\Pb^*$-law for the Skorokhod product topology on $D([0,\infty),\R)\times D([0,\infty),\R)$
    \begingroup
    \addtolength\jot{7pt}
    \begin{align*}
        &\left(\Big(\frac{C_{\infty}}{n^{\kappa\land 2-1}}\bm{H}^{(n)}(\lfloor sn^{\kappa\land 2}\rfloor);\; s\geq 0\Big),\;\Big(\frac{C_{\infty}}{n^{\kappa\land 2}}\bm{F}^{(n)}(\tn);\; t\geq 0\Big)\right) \\ & \underset{n\to\infty}{\longrightarrow}\Big(\big(\bm{H}_{\kappa}(sC_{\infty}/\bm{c}_{\infty});\; s\geq 0\big),\; \big(\bm{c}_{\infty}\tau_{-t}(Y_{\kappa});\; t\geq 0\big)\Big); 
    \end{align*}
    \endgroup
     \item\label{JointConvForest2} if $\kappa=2$, then in $\Pb^*$-law for the Skorokhod product topology on $D([0,\infty),\R)\times D([0,\infty),\R)$
    \begingroup
    \addtolength\jot{7pt}
    \begin{align*}
        &\left(\Big(2C_{\infty}\frac{\log n}{n}\bm{H}^{(n)}(\lfloor sn^2/\log n^2\rfloor);\; s\geq 0\Big);\; \Big(\frac{2C_{\infty}\log n}{n^{2}}\bm{F}^{(n)}(\tn);\; t\geq 0\Big)\right) \\ & \underset{n\to\infty}{\longrightarrow}\Big(\big(\bm{H}_{\kappa}(sC_{\infty}/\bm{c}_{\infty});\; s\geq 0\big);\; \big(\bm{c}_{\infty}\tau_{-t}(Y_{2});\; t\geq 0\big)\Big).
    \end{align*}
    \endgroup
\end{enumerate}
\end{Fact}
\noindent The convergences of $(\frac{C_{\infty}}{n^{\kappa\land 2-1}}\bm{H}^{(n)}(\lfloor sn^{\kappa\land 2}\rfloor);\; s\geq 0)$ when $\kappa\not=2$, of $(2C_{\infty}\frac{\log n}{n}\bm{H}^{(n)}(\lfloor sn^2/\log n^2\rfloor);\; s\geq 0)$ when $\kappa=2$ and of $(\frac{C_{\infty}}{n^{\kappa\land 2}}\bm{F}^{(n)}(\tn);\; t\geq 0)$ when $\kappa\not=2$, of $(\frac{2C_{\infty}\log n}{n^{2}}\bm{F}^{(n)}(\tn);\; t\geq 0)$ when $\kappa=2$ have been obtained separately, the first two in \cite{AidRap} when $\kappa>2$ and in \cite{deRaph1} when $\kappa\in(1,2]$ and the last two in \cite{AK23LocalTimes}. Let us explain briefly why we actually have the joint convergence when $\kappa\in(1,2)$. For any $i\geq 1$ and any $x\in\mathbfcal{R}^{(n)}_i$, define
\begin{align*}
    G^1_i(e_i):=0\;\;\textrm{ and }\;\;\forall\; x\in\mathbfcal{R}^{(n)}_i\setminus\{e_i\},\; G^1_i(x):=\sum_{e_i\leq z<x}\un_{\{\bm{t}^{(n)}_z=1\}}.
\end{align*}
Define the tree $\mathbfcal{T}_i^{(n,1)}$ as follows: for any $k\geq 0$, the generation $k$ of $\mathbfcal{T}_i^{(n,1)}$ is made up of vertices in $x\in\mathbfcal{R}^{(n)}_i$ such that $\bm{t}^{(n)}_x=1$ and $G^1_i(x)=k$ by keeping the initial order on $\mathbfcal{R}^{(n)}_i$ (in the sense that the $(j+1)$-th vertex of $\mathbfcal{R}^{(n)}_i$ such that $\bm{t}^{(n)}_x=1$ is the $(j+1)$-th vertex of $\mathbfcal{T}^{(n,1)}_i$) so that the genealogical structure is preserved. In particular, $\mathbfcal{T}_i^{(n,1)}$ is rooted at $e_i$. By definition, $(\mathbfcal{T}_i^{(n,1)})_{i\geq 1}$ is a collection of\textrm{ i.i.d } random trees. Moreover, if $\bm{N}^{(n,1)}:=\sum_{x\in\mathbfcal{R}^{(n)}_1}\un_{\{\bm{t}^{(n)}_x=1,\; G^1_1(x)=1\}}$ denotes the number of vertices in the generation $1$ of $\mathbfcal{T}_1^{(n,1)}$, then $\E[\bm{N}^{(n,1)}]=1$, $\lim_{r\to\infty}r^{\kappa}\P(\bm{N}^{(n,1)}>r)=c_{\kappa}\in(0,\infty)$ for all $\kappa\in(1,2)$, see \eqref{ConstantTail} and $\mathbfcal{F}^{(n,1)}:=(\mathbfcal{T}^{(n,1)}_i)_{i\geq 1}$ is a forest of\textrm{ i.i.d }Galton-Watson trees with offspring distribution $\bm{N}^{(n,1)}$. If we denote by $\bm{H}^{(n,1)}$ the height function of the forest $\mathbfcal{F}^{(n,1)}$, then it is proved in \cite{deRaph1} that for any $M>0$, in $\P$-probability
\begin{align*}
    \sup_{s\in[0,M]}\frac{1}{n^{\kappa-1}}\Big|\bm{H}^{(n)}(\lfloor sn^{\kappa}\rfloor)-\frac{1}{C_{\infty}}\bm{H}^{(n,1)}\big(\varphi(\lfloor sn^{\kappa}\rfloor)\big)\Big|\underset{n\to\infty}{\longrightarrow}0,
\end{align*}
for some random non-decreasing sequence $(\varphi(m))_{m\geq 1}$ such that in $\P$-probability, $\varphi(m)/m\to C_{\infty}/\bm{c}_{\infty}$. Now, for any $j\geq 0$, denote by $x^1(j)$ the $(j+1)$-th vertex (for the depth first search order) of the forest $\mathbfcal{F}^{(n)}$ such that $\bm{t}^{(n)}_{x^1(j)}=1$. 
For any $x\in\mathbfcal{R}^{(n)}_i$, define the set $\mathbfcal{C}^{(n,1)}_x:=\{z\in\mathbfcal{R}^{(n)}_i;\; G^1_i(z)=G^1_i(x)+1\}$. If we let 
\begin{align*}
    \bm{M}^{(n,1)}_j:=\sum_{y\in\mathbfcal{F}^{(n)}}\un_{\{y>x^1(j)\}}-\sum_{z\in\mathbfcal{C}_{x^1(j)}^{(n,1)}}\;\sum_{y\in\mathbfcal{F}^{(n)}}\un_{\{y>z\}},  
\end{align*}
then the collection $(\bm{M}_j^{(n,1)})_{j\geq 1}$ is made up of\textrm{ i.i.d }random variables and $\E[\bm{M}_1^{(n,1)}]=\bm{c}_{\infty}/C_{\infty}$. We refer to Proposition 2 in \cite{deRaph1} for the computation of $\E[\bm{N}^{(n,1)}]$ and $\E[\bm{M}_1^{(n,1)}]$. In particular, the law of large numbers yields, in $\P$-probability, $\lim_{m\to\infty}\frac{1}{m}\sum_{j=0}^{m-1}\bm{M}_j^{(n,1)}=\bm{c}_{\infty}/C_{\infty}$. In \cite{AK23LocalTimes}, it is noticed that 
\begin{align*}
    \bm{F}^{(n)}(\tn)=\tn+\sum_{j=0}^{\bm{F}^{(n,1)}(\tn)-1}\bm{M}_j^{(n,1)},
\end{align*}
where $\bm{F}^{(n,1)}(\tn)$ is the number of vertices in the first $\tn$ trees of the forest $\mathbfcal{F}^{(n,1)}$. One can note that $\bm{F}^{(n,1)}(\tn)=\inf{\{j\geq 0;\; \bm{V}^{(n,1)}_j=-\tn\}}$ where $\bm{V}^{(n,1)}_0=0$ and for any $j\geq 1$, $$\bm{V}^{(n,1)}_j:=\sum_{i=0}^{j-1}\big(\bm{N}^{(n,1)}_i-1\big),$$ and $\bm{N}^{(n,1)}_i$ is the number of children of the $(i+1)$-th vertex (for the depth first search order) of the random forest $\mathbfcal{F}^{(n,1)}$. In other words, $\bm{V}^{(n,1)}_j$ is a sum of $j$\textrm{ i.i.d }copies of $\bm{N}^{(n,1)}$, $(\bm{V}^{(n,1)}_j)_{j\geq 0}$ denotes the Lukasiewicz path associated with the random forest $\mathbfcal{F}^{(n,1)}$ and Corollary 2.5.2 in \cite{DuqLeGall} yields, in law under $\P$ for the Skorokhod product topology on $D([0,\infty),\R)\times D([0,\infty),\R)$
\begin{align*}
    &\left(\Big(\frac{1}{n^{\kappa-1}}\bm{H}^{(n,1)}(\lfloor sn^{\kappa}\rfloor);\; s\geq 0\Big),\;\Big(\frac{1}{n^{\kappa}}\bm{F}^{(n,1)}(\tn);\; t\geq 0\Big)\right)\underset{n\to\infty}{\longrightarrow}\Big(\big(\bm{H}_{\kappa}(s);\; s\geq 0\big),\; \big(\tau_{-t}(Y_{\kappa});\; t\geq 0\big)\Big),
\end{align*}
thus giving the convergence in Fact \ref{JointConvForest} for the annealed probability $\P$. To obtain the convergence in $\Pb^*$-law, we use the same procedure as in the proof of Theorem 1.1 in \cite{AidRap}. This argument is recalled in the \textbf{Step 3} of the proof of Theorem \ref{Theorem3}.

\vspace{0.2cm}

\noindent Let us now introduce the reduced \textit{multi-type additive martingale}: 
\begin{align}\label{RedAddMart}
    \bm{Z}^{(n)}_{k}(q):=\sum_{i=1}^{q}\bm{Z}^{(n,i)}_k\;\;\textrm{ where }\;\;\bm{Z}^{(n,i)}_k:=\sum_{x\in\mathbfcal{R}^{(n)}_i}\un_{\{|x|_i=k\}}\bm{t}^{(n)}_x,
\end{align}
and for any $x\in\mathbfcal{R}^{(n)}_i$, $|x|_i$ is the generation of the vertex $x$ relative to the tree $\mathbfcal{R}^{(n)}_i$. In particular, $|e_i|_i=0$ and for any $i\in\{1,\ldots,B^{(n)}\}$, $|x|=|x|_i+|e_i|$. Also introduce the following reduced local time 
\begin{align}\label{RedLocalTime}
    \bm{L}^{(n)}_{k}(q):=\sum_{i=1}^{q}\bm{L}^{(n,i)}_k\;\;\textrm{ where }\;\;\bm{L}^{(n,i)}_{k}:=\bm{Z}^{(n,i)}_{k}+\bm{Z}^{(n,i)}_{k+1}.  
\end{align}
The benefit of introducing the reduced range and its associated processes might not be obvious at first but it turns out to be very useful especially to obtain our annealed results. Indeed, we can always write the range $\mathcal{R}^{(n)}$ as the union $\bigcup_{j=1}^n\{x\in\T;\; N_x^{(j)}-N_x^{(j-1)}\geq 1\}$ of\textrm{ i.i.d }copies of $\mathcal{R}^{(1)}$ under the quenched probability $\P^{\mathcal{E}}$ thanks to the strong Markov property. But working under the annealed probability $\P$ is more convenient for the following reason: by Fact \ref{FactGWMulti}, under the annealed probability $\P$, $$(\VectCoord{\mathcal{R}}{1},(\VectCoord{N}{1}_x;\; x\in\VectCoord{\mathcal{R}}{1}))$$ is a multi-type Galton-Watson tree (but in general, this is not the true under $\P^{\mathcal{E}}$). Moreover, the trees $\{x\in\T;\; N_x^{(j)}-N_x^{(j-1)}\geq 1\}$ for $j=1,\ldots, n$ are (in general) not\textrm{ i.i.d copies of }$\mathcal{R}^{(1)}$ under $\P$. In order to fix this, that is be able to take advantage of both independence and the properties of multi-type Galton-Watson trees under the annealed probability $\P$, the reduced range is introduced: $((\mathbfcal{R}^{(n)}_i,\; (\bm{t}_x^{(n)};\; x\in\mathbfcal{R}^{(n)}_i));\; i\geq 1)$ is a collection of\textrm{ i.i.d }copies of $(\mathcal{R}^{(1)},\; (N_x^{(1)},\; x\in\mathcal{R}^{(1)}))$ under $\P$. In particular, $\bm{Z}^{(n)}_{k}(q)$ (resp. $\bm{L}^{(n)}_{k}(q)$) is a sum of $q$\textrm{ i.i.d }copies of $Z^{(1)}_{k}$ (resp. $L^{(1)}_{k}$) under the annealed probability $\P$ (but as we said, in general, $Z^{(p)}_{k}$ (resp. $L^{(p)}_{k}$) is a not sum of $p$\textrm{ i.i.d }copies of $Z^{(1)}_{k}$ (resp. $L^{(1)}_{k}$) under the annealed probability $\P$).

\vspace{0.2cm}

\noindent We are now ready to prove our first main result.

\subsection{Proof of Theorem \ref{Theorem1}}\label{ProofTheorem1}

We proceed in two steps. The first step is dedicated to the estimation of the quenched probability $\P^{\mathcal{E}}(L^{(1)}_m>0)$. For that, we use the reduced counterpart of the local time and make a link with the height process associated to the range and conclude with an excursion theory argument. The second step is devoted to the annealed estimation of $\P(L^{(1)}_m>0)$. For that, we use the first step and that $\bm{L}^{(n)}_{k}(q)$ is a sum of $q$\textrm{ i.i.d }copies of $L^{(1)}_{k}$ under $\P$.

\vspace{0.2cm}

\noindent\textbf{Step 1:} we provide an equivalent for $\P^{\mathcal{E}}(L^{(1)}_m>0)$.

\vspace{0.1cm}

\noindent First assume that $\kappa\not=2$. One can see that a direct consequence of the fact that $$\P^{\mathcal{E}}\Big(\sup_{x\in\Bar{\mathcal{B}}^{(n)}}|x|\leq\lfloor(\log n)^3\rfloor;\; \forall\; |y|\geq\lfloor(\log n)^3\rfloor,\exists\; x\in\Bar{\mathcal{B}}^{(n)}:\; x\leq y,\; \Bar{\mathcal{B}}^{(n)}\subset\mathcal{B}^{(n)}\Big)$$ goes to $1$ for almost every environment is that, $\Pb^*$-almost surely as $n\to\infty$
\begin{align}\label{ApproxReduced}
    \P^{\mathcal{E}}\Big(\forall\;a\geq 0,\; L^{(n)}_{\lfloor an^{\kappa\land 2-1}\rfloor}=\sum_{i=1}^{\CardRoots}\bm{L}^{(n,i)}_{\lfloor an^{\kappa\land 2-1}\rfloor-|e_i|}\Big)\underset{n\to\infty}{\longrightarrow}1.
\end{align}
Hence, by \eqref{ApproxReduced}, using again that $\Pb^*$-almost surely, $$\lim_{n\to\infty}\P^{\mathcal{E}}(\forall\; 1\leq i\leq\CardRoots;\;\lfloor(\log n)^2\rfloor\leq |e_i|\leq\lfloor(\log n)^3\rfloor)=1,$$ together with the fact that $\{\bm{L}^{(n)}_{k-k^+}(\CardRoots)>0\}\subset\{\sum_{i=1}^{\CardRoots}\bm{L}^{(n,i)}_{k-|e_i|}>0\}\subset\{\bm{L}^{(n)}_{k-k^-}(\CardRoots)>0\}$ for any $k^-\leq|e_i|\leq k^+\leq k$, we obtain
\begin{align*}
    \P^{\mathcal{E}}\Big(L^{(n)}_{\lfloor n^{\kappa\land 2-1}\rfloor}>0\Big)\leq\P^{\mathcal{E}}\Big(\bm{L}^{(n)}_{\lfloor n^{\kappa\land 2-1}\rfloor-\lfloor(\log n)^3\rfloor}\big(\CardRoots\big)>0\Big)+o^{\mathcal{E}}(1), 
\end{align*}
and
\begin{align*}
    \P^{\mathcal{E}}\Big(L^{(n)}_{\lfloor n^{\kappa\land 2-1}\rfloor}>0\Big)\geq\P^{\mathcal{E}}\Big(\bm{L}^{(n)}_{\lfloor n^{\kappa\land 2-1}\rfloor-\lfloor(\log n)^2\rfloor}\big(\CardRoots\Big)>0\big)+o^{\mathcal{E}}(1), 
\end{align*}
where $o^{\mathcal{E}}(1)$ represent any sequence converging to $0$, $\Pb^*$-almost surely as $n\to\infty$. Let $(a_n)$ be a sequence of positive integers such that $a_n/n^{\kappa\land 2-1}\to 1$. We have
\begin{align*}
    \P^{\mathcal{E}}\Big(\bm{L}^{(n)}_{a_n}\big(\CardRoots\big)>0\Big)=\P^{\mathcal{E}}\Big(\max_{1\leq i\leq\CardRoots}\bm{L}^{(n,i)}_{a_n}>0\Big)&=\P^{\mathcal{E}}\Big(\max_{1\leq i\leq\CardRoots}\max_{x\in\mathbfcal{R}^{(n)}_i}|x|_i\geq a_n\Big) \\ & =\P^{\mathcal{E}}\Big(\max_{1\leq j\leq \bm{F}^{(n)}(\CardRoots)}\bm{H}^{(n)}(j)\geq a_n\Big).
\end{align*}
Thanks to Fact \ref{JointConvForest} \ref{JointConvForest1} and \eqref{ConvCardRoot} (noticing that $W_{\infty}$ is deterministic at fixed environment), we have, in $\Pb^*$-law on $D([0,\infty))\times\R$
\begin{align*}
    \left(\Big(\frac{C_{\infty}}{n^{\kappa\land 2-1}}\bm{H}^{(n)}(\lfloor sn^{\kappa\land 2}\rfloor);\; s\geq 0\Big),\;\frac{C_{\infty}}{n^{\kappa\land 2}}\bm{F}^{(n)}\big(\CardRoots\big)\right)\underset{n\to\infty}{\longrightarrow}\Big(\big(\bm{H}_{\kappa}(sC_{\infty}/\bm{c}_{\infty});\; s\geq 0\big),\;\bm{c}_{\infty}\tau_{-W_{\infty}}(Y_{\kappa})\Big), 
\end{align*}
thus giving, in $\Pb^*$-probability
\begin{align*}
    \P^{\mathcal{E}}\Big(\max_{1\leq j\leq \bm{F}^{(n)}(\CardRoots)}\bm{H}^{(n)}(j)\geq a_n\Big)\underset{n\to\infty}{\longrightarrow}\P^{\mathcal{E}}\Big(\sup_{t\leq\tau_{-W_{\infty}}(Y_{\kappa})}\bm{H}_{\kappa}(t)>C_{\infty}\Big)=1-e^{-W_{\infty}v(C_{\infty})},
\end{align*}
where we have used excursion theory for the last equality and $v(a)$ satisfies $\int_{v(a)}^{\infty}\mathrm{d}\lambda/(\bm{\mathrm{C}}_{\kappa}\lambda^{\kappa\land 2})=a$. We refer to Chapter VIII.2 in \cite{LivreBertoin} for details about excursions theory and Corollary 1.4.2 in \cite{DuqLeGall} for details about $v(\cdot)$. In other words, $v(C_{\infty})=((\kappa\land 2-1)C_{\infty}\bm{\mathrm{C}}_{\kappa})^{-1/(\kappa\land 2-1)}$ thus giving \\(taking $a_n=\lfloor n^{\kappa\land 2-1}\rfloor-\lfloor(\log n)^3\rfloor$ in a first time and $a_n=\lfloor n^{\kappa\land 2-1}\rfloor-\lfloor(\log n)^2\rfloor$ in a second time), in $\Pb^*$-probability
\begin{align}\label{ConvSurvivalProba}
     \lim_{n\to\infty}\P^{\mathcal{E}}\Big(L^{(n)}_{\lfloor n^{\kappa\land 2-1}\rfloor}>0\Big)=1-\exp\Big(-W_{\infty}\big((\kappa\land 2-1)C_{\infty}\bm{\mathrm{C}}_{\kappa}\big)^{-1/(\kappa\land 2-1)}\Big).
\end{align}
On the other hand, using that under $\P^{\mathcal{E}}$, $\VectCoord{Z}{p}_k$ is a sum of $p$\textrm{ i.i.d }random variables distributed as $\VectCoord{Z}{1}_k$, we obtain
\begin{align*}
    \P^{\mathcal{E}}\Big(\VectCoord{L}{n}_{\lfloor n^{\kappa\land 2-1}\rfloor}>0\Big)=1-\Big(1-\P^{\mathcal{E}}\big(\VectCoord{Z}{1}_{\lfloor n^{\kappa\land 2-1}\rfloor}>0\big)\Big)^{n},
\end{align*}
and thanks to \eqref{ConvSurvivalProba}, we have, in $\Pb^*$-probability
\begin{align}\label{SurvivalProbaQuenched}
    n\P^{\mathcal{E}}\big(\VectCoord{Z}{1}_{\lfloor n^{\kappa\land 2-1}\rfloor}>0\big)\underset{n\to\infty}{\longrightarrow}W_{\infty}\big((\kappa\land 2-1)C_{\infty}\bm{\mathrm{C}}_{\kappa}\big)^{-1/(\kappa\land 2-1)}.
\end{align}
We conclude by recalling the definition of $\bm{\mathrm{C}}_{\kappa}$ in \eqref{DefCkappa}. Now assume that $\kappa=2$. Similarly, thanks to Fact \ref{JointConvForest} \ref{JointConvForest2}
\begin{align*}
    \P^{\mathcal{E}}\Big(\VectCoord{L}{n}_{\lfloor n/\log n\rfloor}>0\Big)\underset{n\to\infty}{\longrightarrow}\P^{\mathcal{E}}\Big(\sup_{t\leq\tau_{-W_{\infty}}(Y_{\kappa})}\bm{H}_{\kappa}(t)>2C_{\infty}\Big)=1-e^{-W_{\infty}v(2C_{\infty})}=1-e^{-W_{\infty}/(C_{\infty}c_2)},
\end{align*}
where we have used that $\bm{\mathrm{C}}_2=c_2/2$, see \eqref{DefCkappa}. This yields, in $\Pb^*$-probability
\begin{align*}
    n\P^{\mathcal{E}}\big(\VectCoord{L}{1}_{\lfloor n/\log n\rfloor}>0\big)\underset{n\to\infty}{\longrightarrow}\frac{W_{\infty}}{C_{\infty}c_2},
\end{align*}
and this ends the first step.

\vspace{0.2cm}

\noindent\textbf{Step 2:} we provide an equivalent for $\P(L^{(1)}_n>0)$.

\vspace{0.1cm}

\noindent First assume that $\kappa\not=2$. We claim that
\begin{align}\label{SurvivalProbaAnnealed}
    n\P\big(\VectCoord{L}{1}_{\lfloor n^{\kappa\land 2-1}\rfloor}>0\big)\underset{n\to\infty}{\longrightarrow}\big((\kappa\land 2-1)C_{\infty}\bm{\mathrm{C}}_{\kappa}\big)^{-1/(\kappa\land 2-1)}.
\end{align}
Indeed, on the one hand, thanks to \eqref{ConvSurvivalProba}
\begin{align}\label{ConvProbaPositive}
     \P^*\Big(\bm{L}^{(n)}_{\lfloor n^{\kappa\land 2-1}\rfloor}\big(\CardRoots\big)>0\Big)\underset{n\to\infty}{\longrightarrow}\Eb^*\Big[1-\exp\Big(-W_{\infty}\big((\kappa\land 2-1)C_{\infty}\bm{\mathrm{C}}_{\kappa}\big)^{-1/(\kappa\land 2-1)}\Big)\Big]\in(0,1).
\end{align}
On the other hand, by definition of $\bm{L}^{(n)}_k(\CardRoots)$
\begin{align*}
    \P\Big(\bm{L}^{(n)}_{\lfloor n^{\kappa\land 2-1}\rfloor}\big(\CardRoots\big)>0\Big)&=1-\P\Big(\forall\; 1\leq i\leq\Bar{B}^{(n)}:\;\bm{L}^{(n,i)}_{\lfloor n^{\kappa\land 2-1}\rfloor}=0,\; \Bar{B}^{(n)}>0\Big)-\P\big(\Bar{B}^{(n)}=0\big) \\ & =\E\Bigg[1-\Big(1-\P\big(L^{(1)}_{\lfloor n^{\kappa\land 2-1}\rfloor}>0\big)\Big)^{\CardRoots}\un_{\{\Bar{B}^{(n)}>0\}}\Bigg]-\P\big(\Bar{B}^{(n)}=0\big) \\ & =\E\Bigg[1-\Big(1-\P\big(L^{(1)}_{\lfloor n^{\kappa\land 2-1}\rfloor}>0\big)\Big)^{\CardRoots}\Bigg].
\end{align*}
Besides, using that $\lim_{n\to\infty}(\Bar{B}^{(n)}>0,\; \textrm{extinction of }\T)=0$ yields, as $n\to\infty$
\begin{align*}
    \P^*\Big(\bm{L}^{(n)}_{\lfloor n^{\kappa\land 2-1}\rfloor}\big(\CardRoots\big)>0\Big)=\E^*\Bigg[1-\Big(1-\P\big(L^{(1)}_{\lfloor n^{\kappa\land 2-1}\rfloor}>0\big)\Big)^{\CardRoots}\Bigg]+o(1).
\end{align*}
By \eqref{ConvCardRoot}, we have that $\CardRoots/n\to W_{\infty}$ in $\P^*$-probability as $n\to\infty$ so thanks to \eqref{ConvProbaPositive}, we deduce that the sequence $(n\P(L^{(1)}_{\lfloor n^{\kappa\land 2-1}\rfloor}>0))_{n\geq 1}$ is bounded and if we denote by $\zeta\in(0,1]$ the limit of a sub-sequence, we then also deduce that
\begin{align*}
    \Eb^*\Big[\exp\Big(-W_{\infty}\big((\kappa\land 2-1)C_{\infty}\bm{\mathrm{C}}_{\kappa}\big)^{-1/(\kappa\land 2-1)}\Big)\Big]=\Eb^*\big[\exp\big(-\zeta W_{\infty}\big)\big].
\end{align*}
Finally, since $\Pb^*(W_{\infty}>0)=1$, the function $\theta\in(0,\infty)\mapsto \Eb^*[\exp(-\theta W_{\infty})]$ is injective and in particular, we necessarily have $\zeta=((\kappa\land 2-1)C_{\infty}\bm{\mathrm{C}}_{\kappa})^{-1/(\kappa\land 2-1)}$ and this is what we wanted. \\
The exact same strategy for $\kappa=2$ gives
\begin{align*}
    n\P\big(\VectCoord{L}{1}_{\lfloor n/\log n\rfloor}>0\big)\underset{n\to\infty}{\longrightarrow}\frac{1}{C_{\infty}c_2},
\end{align*}
and the second step is completed. \\
Finally, combining \textbf{Step 1} and \textbf{Step 2} yields, for any $\kappa>1$, in $\Pb^*$-probability
\begin{align*}
    \frac{\P^{\mathcal{E}}\big(\VectCoord{L}{1}_{m}>0\big)}{\P\big(\VectCoord{L}{1}_{m}>0\big)}\underset{m\to\infty}{\longrightarrow}W_{\infty},
\end{align*}
which ends the proof of our first theorem. \\
We now prove our important proposition which provide an estimation of the local time in critical generations up to the first return time to the root.

\subsection{Proof of Proposition \ref{UnifLawMultyMart}}\label{ProofsOfPropositions} 

Let us begin with a collection of lemmas that will be useful for the proof of our proposition.

\subsubsection{Preliminary results}\label{SectionPR}

For any $s\in[0,1]$ and any $p,k\geq 1$, introduce $G^{\mathcal{E}}_{k}(s,p):=\E^{\mathcal{E}}[s^{Z^{(p)}_k}]$, the quenched probability generating function of $Z^{(p)}_k$ and $G_{k}(s,p):=\Eb[G_k(s,p)]=\E[s^{Z^{(p)}_k}]$ its annealed version, $G^{\mathcal{E}}_{k}(s):=G^{\mathcal{E}}_{k}(s,1)$ and $G_{k}(s):=G_{k}(s,1)$. In the next lemma, we obtain lower and upper bounds for $G_{\ell}(s,p)$ which allows in particular to deduce a result for $G_{k+\ell}(s)$:
\begin{lemm}\label{FoncGen}
Let $p,k,\ell\geq 1$ be integers, $s\in[0,1]$ and $\alpha\in(1,\kappa\land 2)$. We have
\begin{align*}
    \big(G_{\ell}(s)\big)^p\leq G_{\ell}(s,p)\leq 1\land e^{-p\big(1-G_{\ell}(s)\big)+p^{\alpha}\Eb\big[\big(1-G^{\mathcal{E}}_{\ell}(s)\big)^{\alpha}\big]},
\end{align*}
and in particular
\begin{align*}
    G_{k}(G_{\ell}(s))\leq G_{k+\ell}(s)\leq\E\left[1\land e^{-Z^{(1)}_k(1-G_{\ell}(s))}e^{\sum_{|u|=k}\big(N^{(1)}_u\big)^{\alpha}\Eb\big[\big(1-G^{\mathcal{E}}_{\ell}(s)\big)^{\alpha}\big]}\right].
\end{align*}
\end{lemm}
\begin{proof}
Since $Z^{(p)}_{\ell}$ is a sum of\textrm{ i.i.d }copies of $Z^{(1)}_{\ell}$ under the quenched probability $\P^{\mathcal{E}}$, the Jensen inequality yields $G_{\ell}(s,p)=\Eb[(G^{\mathcal{E}}_{\ell}(s))^p]\geq (G_{\ell}(s))^p$, thus giving the lower bound. For the upper bound, we have $G_{\ell}(s,p)\leq 1$ by definition. Then, one can notice that using the inequality $t\leq e^{t-1}$ for any $t\in\R$ first and the fact that for any $s\geq 0$, $e^{-s}\leq 1-s+s^{\alpha}$ in a second time, we have
\begin{align*}
    G^{\mathcal{E}}_{\ell}(s,p)=\big(G^{\mathcal{E}}_{\ell}(s)\big)^p\leq e^{-p(1-G^{\mathcal{E}}_{\ell}(s))}\leq 1-p\big(1-G^{\mathcal{E}}_{\ell}(s)\big)+p^{\alpha}\big(1-G^{\mathcal{E}}_{\ell}(s)\big)^{\alpha}.
\end{align*}
Then taking the expectation and using the fact that $1+t\leq e^t$ for any $t\in\R$, we obtain
\begin{align*}
    \Eb\big[G^{\mathcal{E}}_{\ell}(s,p)\big]\leq 1-p\big(1-G_{\ell}(s)\big)+p^{\alpha}\Eb\big[\big(1-G^{\mathcal{E}}_{\ell}(s)\big)^{\alpha}\big]\leq e^{-p\big(1-G_{\ell}(s)\big)+p^{\alpha}\Eb\big[\big(1-G^{\mathcal{E}}_{\ell}(s)\big)^{\alpha}\big]},
\end{align*}
and this yields the first result. Now, by the branching property (see Fact \ref{FactGWMulti}), we have
\begin{align*}
    G_{k+\ell}(s)=\E\Big[\prod_{u\in\mathcal{R}^{(1)};|u|=k}G_{\ell}\big(s,N_u^{(1)}\big)\Big],
\end{align*}
so using the inequality we have just proved with $p=N_u^{(1)}\geq 1$ for $u\in\mathcal{R}^{(1)}$, we obtain for the lower bound 
\begin{align*}
    G_{k}\big(G_{\ell}(s)\big)=\E\Big[\big(G_{\ell}(s)\big)^{Z^{(1)}_{k}}\Big]=\E\Big[\prod_{u\in\mathcal{R}^{(1)};|u|=k}G_{\ell}\big(s\big)^{N_u^{(1)}}\Big]\leq\E\Big[\prod_{u\in\mathcal{R}^{(1)};|u|=k}G_{\ell}\big(s,N_u^{(1)}\big)\Big]= G_{k+\ell}(s)
\end{align*}
and for the upper bound 
\begin{align*}
    G_{k+\ell}(s)=\E\Big[\prod_{u\in\mathcal{R}^{(1)};|u|=k}G_{\ell}\big(s,N_u^{(1)}\big)\Big]&\leq\E\Big[\prod_{u\in\mathcal{R}^{(1)};|u|=k}\Big(1\land e^{-N_u^{(1)}\big(1-G_{\ell}(s)\big)+\big(N_u^{(1)}\big)^{\alpha}\Eb\big[\big(1-G^{\mathcal{E}}_{\ell}(s)\big)^{\alpha}\big]}\Big)\Big] \\ & \leq\E\left[1\land e^{-Z^{(1)}_k(1-G_{\ell}(s))}e^{\sum_{|u|=k}\big(N^{(1)}_u\big)^{\alpha}\Eb\big[\big(1-G^{\mathcal{E}}_{\ell}(s)\big)^{\alpha}\big]}\right],
\end{align*}
and the proof is completed.
\end{proof}

\begin{lemm}\label{NewEquivFoncGen}
Let $i\geq 1$ be an integer and let
\begin{itemize}
    \item $(\gamma_{i,m})_{m\geq 1}$ and $(\gamma'_{i,m})_{m\geq 1}$ be two sequences of positive integers such that $\gamma_{i,m}\leq\gamma'_{i,m}$, $\gamma_{i,m}\to\infty$ and $\gamma'_{i,m}/\gamma_{i,m}\to 1$ as $m\to\infty$, uniformly in $i\geq 1$;
    \item $(t_{i,m})$ and $(t'_{i,m})$ be two sequences of positive real numbers such that $\sup_{i,m\geq 1}t_{i,m}/p_{\gamma_{i,m}}<\infty$ and \\ $\sup_{i,m\geq 1}t'_{i,m}/p_{\gamma'_{i,m}}<\infty$;
    \item $(s_{i,m})$ be a sequence of positive real numbers such that $t_{i,m}/s_{i,m}\to 1$ as $m\to\infty$ uniformly in $i\geq 1$;
    \item $(r_{i,m})_{m\geq 1}$ be a sequence of non-negative numbers such that $\sup_{i\geq 1}r_{i,m}/p_{\gamma_{i,m}}\to 0$ as $m\to\infty$ and uniformly in $i\geq 1$;
    \item $(z_{i,m})$ be a sequence of positive numbers such that $p_{\gamma_{i,m}}/z_{i,m}\to\mu\in(0,\infty)$ as $m\to\infty$, uniformly in $i\geq 1$
\end{itemize}
We have
\begin{align}\label{LemmConv1}
    \lim_{m\to\infty}\left(\frac{\E\Big[1\land e^{-t_{i,m}Z^{(1)}_{\gamma_{i,m}}-t'_{i,m}Z^{(1)}_{\gamma'_{i,m}}+r_{i,m}\sum_{|u|=\gamma'_{i,m}}\big(N^{(1)}_u\big)^{\alpha}}\Big]}{\E\Big[e^{-t_{i,m}Z^{(1)}_{\gamma_{i,m}}-t'_{i,m}Z^{(1)}_{\gamma'_{i,m}}}\Big]}\right)^{1/z_{i,m}}=1,
\end{align}
and
\begin{align}\label{LemmConv2}
    \lim_{m\to\infty}\left(\frac{\E\Big[1\land e^{-t_{i,m}Z^{(1)}_{\gamma_{i,m}}+r_{i,m}\sum_{|u|=\gamma_{i,m}}\big(N^{(1)}_u\big)^{\alpha}}\Big]}{\E\Big[e^{-s_{i,m}Z^{(1)}_{\gamma_{i,m}}}\Big]}\right)^{1/z_{i,m}}=1,
\end{align}
uniformly in $i\geq 1$.
\end{lemm}
\begin{proof}
The proofs of \eqref{LemmConv1} and \eqref{LemmConv2} are the same so we only deal with \eqref{LemmConv1}. Let us first prove that
\begin{align}\label{UnifEspBound}
    \sup_{k\geq 0}\E\Big[\sum_{|x|=k}\big(N_x^{(1)}\big)^{\alpha}\Big]<\infty.
\end{align}
Indeed, recall that for any $x\in\T$, $N_x^{(1)}$ is, under $\P^{\mathcal{E}}_{x^*}$, a Geometric random variable on $\N$ with probability of success $1-\P^{\mathcal{E}}_{x^*}(N^{(1)}_x\geq 1)$, see section \ref{SectionPR} so the strong Markov property yields $$\P^{\mathcal{E}}\big(N_x^{(1)}>\ell\big)=\P^{\mathcal{E}}\big(N^{(1)}_x\geq 1\big)\big(\P^{\mathcal{E}}_{x^*}\big(N^{(1)}_x\geq 1\big)\big)^{\ell}=\frac{e^{-V(x)}}{H_x}\Big(1-\frac{1}{H_x}\Big)^{\ell}.$$ Hence, there exists a constant $C_{\alpha}>0$ such that
\begin{align*}
    \E\Big[\sum_{|x|=k}\big(N_x^{(1)}\big)^{\alpha}\Big]&\leq C_{\alpha}\Eb\Big[\sum_{|x|=k}\frac{e^{-V(x)}}{H_x}\sum_{\ell\geq 0}\ell^{\alpha-1}\Big(1-\frac{1}{H_x}\Big)^{\ell}\Big] \\ & =C_{\alpha}\Eb\Big[\sum_{|x|=k}\big(H_x\big)^{\alpha-1}e^{-V(x)}\frac{1}{(H_x)^{\alpha}}\sum_{\ell\geq 0}\ell^{\alpha-1}\Big(1-\frac{1}{H_x}\Big)^{\ell}\Big] \\ & \leq C_{\alpha}\sup_{l\geq 0}\Eb\Big[\sum_{|x|=l}\big(H_x\big)^{\alpha-1}e^{-V(x)}\Big]\sup_{s\in(0,1]}s^{\alpha}\sum_{\ell\geq 0}\ell^{\alpha-1}(1-s)^{\ell},
\end{align*}
where we have used that $H_x\geq 1$ for all $x\in\T$ for the last inequality. By Lemma 2.2 in \cite{AndDiel3}, we have $\sup_{l\geq 0}\Eb[\sum_{|x|=l}(H_x)^{\alpha-1}e^{-V(x)}]<\infty$ and it is a well-known fact that $\sup_{s\in(0,1]}s^{\alpha}\sum_{\ell\geq 0}\ell^{\alpha-1}(1-s)^{\ell}<\infty$, so we get what we wanted. We now prove \eqref{LemmConv1}. Let us first show that
\begin{align}\label{RatioCondi}
    \lim_{m\to\infty}\frac{\E\bigg[1\land e^{-t_{i,m}Z^{(1)}_{\gamma_{i,m}}-t'_{i,m}Z^{(1)}_{\gamma'_{i,m}}+r_{i,m}\sum_{|u|=\gamma'_{i,m}}\big(N^{(1)}_u\big)^{\alpha}}\Big|Z^{(1)}_{\gamma_{i,m}}>0\bigg]}{\E\bigg[e^{-t_{i,m}Z^{(1)}_{\gamma_{i,m}}-t'_{i,m}Z^{(1)}_{\gamma'_{i,m}}}\Big|Z^{(1)}_{\gamma_{i,m}}>0\bigg]}=1,
\end{align}
uniformly in $i\geq 1$. For that, one can decompose the expectation in the denominator on the events $\{r_{i,m}\sum_{|u|=\gamma'_{i,m}}(N^{(1)}_u)^{\alpha}>\varepsilon\}$ and $\{r_{i,m}\sum_{|u|=\gamma'_{i,m}}(N^{(1)}_u)^{\alpha}\leq\varepsilon\}$, with $\varepsilon>0$. This gives in particular
\begin{align*}
    &\E\bigg[1\land e^{-t_{i,m}Z^{(1)}_{\gamma_{i,m}}-t'_{i,m}Z^{(1)}_{\gamma'_{i,m}}+r_{i,m}\sum_{|u|=\gamma'_{i,m}}\big(N^{(1)}_u\big)^{\alpha}}\Big|Z^{(1)}_{\gamma_{i,m}}>0\bigg] \\ & \leq \E\bigg[e^{-t_{i,m}Z^{(1)}_{\gamma_{i,m}}-t'_{i,m}Z^{(1)}_{\gamma'_{i,m}}+\varepsilon}\Big|Z^{(1)}_{\gamma_{i,m}}>0\bigg]+\P\bigg(r_{i,m}\sum_{|u|=\gamma'_{i,m}}\big(N^{(1)}_u\big)^{\alpha}>\varepsilon\Big|Z_{\gamma_{i,m}}>0\bigg) \\ & \leq e^{\varepsilon}\E\bigg[e^{-t_{i,m}Z^{(1)}_{\gamma_{i,m}}-t'_{i,m}Z^{(1)}_{\gamma'_{i,m}}}\Big|Z^{(1)}_{\gamma_{i,m}}>0\bigg]+\frac{1}{\varepsilon}\sup_{i\geq 1}\E\bigg[\frac{r_{i,m}}{p_{\gamma_{i,m}}}\sum_{|u|=\gamma'_{i,m}}\big(N^{(1)}_u\big)^{\alpha}\bigg],
\end{align*}
where we have used the Markov inequality for the second term in the last inequality. This gives
\begin{align*}
    \frac{\E\bigg[1\land e^{-t_{i,m}Z^{(1)}_{\gamma_{i,m}}-t'_{i,m}Z^{(1)}_{\gamma'_{i,m}}+r_{i,m}\sum_{|u|=\gamma'_{i,m}}\big(N^{(1)}_u\big)^{\alpha}}\Big|Z^{(1)}_{\gamma_{i,m}}>0\bigg]}{\E\bigg[e^{-t_{i,m}Z^{(1)}_{\gamma_{i,m}}-t'_{i,m}Z^{(1)}_{\gamma'_{i,m}}}\Big|Z^{(1)}_{\gamma_{i,m}}>0\bigg]}\leq e^{\varepsilon}+\frac{\sup_{i\geq 1}\E\bigg[\frac{r_{i,m}}{p_{\gamma_{i,m}}}\sum_{|u|=\gamma'_{i,m}}\big(N^{(1)}_u\big)^{\alpha}\bigg]}{\varepsilon\E\bigg[e^{-t_{i,m}Z^{(1)}_{\gamma_{i,m}}-t'_{i,m}Z^{(1)}_{\gamma'_{i,m}}}\Big|Z^{(1)}_{\gamma_{i,m}}>0\bigg]}.
\end{align*}
We want a lower bound for the expectation in the denominator. Let $M=2\sup_{i,m\geq 1}(t_{i,m}+t'_{i,m})/p_{\gamma_{i,m}}$. Note that $M\in(0,\infty)$. Indeed
\begin{align*}
    \sup_{i,m\geq 1}\frac{t_{{i,m}}+t'_{{i,m}}}{p_{\gamma_{i,m}}}\leq \sup_{i,m\geq 1}\frac{t_{{i,m}}}{p_{\gamma_{i,m}}}+\sup_{i,m\geq 1}\frac{t'_{{i,m}}}{p_{\gamma'_{i,m}}}\times\sup_{i,m\geq 1}\frac{p_{\gamma'_{i,m}}}{p_{\gamma_{i,m}}}.
\end{align*}
As we have assumed, $\sup_{i,m\geq 1}t_{{i,m}}/p_{\gamma_{i,m}}<\infty$ and $\sup_{i,m\geq 1}t'_{{i,m}}/p'_{\gamma_{i,m}}<\infty$. Moreover, we have by Theorem \ref{Theorem1} that $k^{1/(\kappa\land 2-1)}p_k$ goes to $((\kappa\land 2-1)C_{\infty}\bm{\mathrm{C}}_{\kappa})^{-1/(\kappa\land 2-1)}$ as $k\to\infty$ when $\kappa\not=2$ and $(k\log k)p_k$ goes to $(C_{\infty}\bm{\mathrm{C}}_{2})^{-1}$ when $\kappa=2$ as $k\to\infty$. In particular, since $\gamma'_{i,m}/\gamma_{i,m}\to 1$ as $m\to\infty$, uniformly in $i\geq 1$, we have that $\sup_{i,m\geq 1}p_{\gamma'_{i,m_{\ell}}}/p_{\gamma_{i,m_{\ell}}}<\infty$, that is $M<\infty$. Note that $t_{i,m}>0$ and $t'_{i,m}>0$ so $M>0$. Now, one can see that
\begin{align*}
    \E\bigg[e^{-t_{i,m}Z^{(1)}_{\gamma_{i,m}}-t'_{i,m}Z^{(1)}_{\gamma'_{i,m}}}\Big|Z^{(1)}_{\gamma_{i,m}}>0\bigg]&\geq\E\bigg[e^{-t_{i,m}Z^{(1)}_{\gamma_{i,m}}-t'_{i,m}Z^{(1)}_{\gamma'_{i,m}}}\un_{\big\{t_{i,m}Z^{(1)}_{\gamma_{i,m}}+t'_{i,m}Z^{(1)}_{\gamma'_{i,m}}\leq M\big\}}\Big|Z^{(1)}_{\gamma_{i,m}}>0\bigg] \\ & \geq e^{-M}\P\Big(t_{i,m}Z^{(1)}_{\gamma_{i,m}}+t'_{i,m}Z^{(1)}_{\gamma'_{i,m}}\leq M\big|Z^{(1)}_{\gamma_{i,m}}>0\Big),
\end{align*}
and thanks to the Markov inequality
\begin{align*}
    \P\Big(t_{i,m}Z^{(1)}_{\gamma_{i,m}}+t'_{i,m}Z^{(1)}_{\gamma'_{i,m}}\leq M\big|Z^{(1)}_{\gamma_{i,m}}>0\Big)\geq 1-\frac{1}{M}\frac{t_{i,m}+t'_{i,m}}{p_{\gamma_{i,m}}}\geq\frac{1}{2},
\end{align*}
that is
\begin{align*}
    \E\bigg[e^{-t_{i,m}Z^{(1)}_{\gamma_{i,m}}-t'_{i,m}Z^{(1)}_{\gamma'_{i,m}}}\Big|Z^{(1)}_{\gamma_{i,m}}>0\bigg]\geq\frac{e^{-M}}{2}.
\end{align*}
This yields
\begin{align*}
    \frac{\E\bigg[1\land e^{-t_{i,m}Z^{(1)}_{\gamma_{i,m}}-t'_{i,m}Z^{(1)}_{\gamma'_{i,m}}+r_{i,m}\sum_{|u|=\gamma'_{i,m}}\big(N^{(1)}_u\big)^{\alpha}}\Big|Z^{(1)}_{\gamma_{i,m}}>0\bigg]}{\E\bigg[e^{-t_{i,m}Z^{(1)}_{\gamma_{i,m}}-t'_{i,m}Z^{(1)}_{\gamma'_{i,m}}}\Big|Z^{(1)}_{\gamma_{i,m}}>0\bigg]}\leq e^{\varepsilon}+\frac{2}{\varepsilon}e^M\sup_{i\geq 1}\E\bigg[\frac{r_{i,m}}{p_{\gamma_{i,m}}}\sum_{|u|=\gamma'_{i,m}}\big(N^{(1)}_u\big)^{\alpha}\bigg].
\end{align*}
Thanks to \eqref{UnifEspBound}, we have
\begin{align*}
    \sup_{i\geq 1}\E\bigg[\frac{r_{i,m}}{p_{\gamma_{i,m}}}\sum_{|u|=\gamma'_{i,m}}\big(N^{(1)}_u\big)^{\alpha}\bigg]\leq\sup_{k\geq 0}\E\Big[\sum_{|x|=k}\big(N_x^{(1)}\big)^{\alpha}\Big]\sup_{i\geq 1}\frac{r_{i,m}}{p_{\gamma_{i,m}}}\underset{m\to\infty}{\longrightarrow}0,
\end{align*}
so for any $\varepsilon>0$
\begin{align*}
    \limsup_{m\to\infty}\frac{\E\bigg[1\land e^{-t_{i,m}Z^{(1)}_{\gamma_{i,m}}-t'_{i,m}Z^{(1)}_{\gamma'_{i,m}}+r_{i,m}\sum_{|u|=\gamma'_{i,m}}\big(N^{(1)}_u\big)^{\alpha}}\Big|Z^{(1)}_{\gamma_{i,m}}>0\bigg]}{\E\bigg[e^{-t_{i,m}Z^{(1)}_{\gamma_{i,m}}-t'_{i,m}Z^{(1)}_{\gamma'_{i,m}}}\Big|Z^{(1)}_{\gamma_{i,m}}>0\bigg]}\leq e^{\varepsilon},
\end{align*}
uniformly in $i\geq 1$. For the lower bound, one can see that
\begin{align*}
    1\land e^{-t_{i,m}Z^{(1)}_{\gamma_{i,m}}-t'_{i,m}Z^{(1)}_{\gamma'_{i,m}}+r_{i,m}\sum_{|u|=\gamma'_{i,m}}\big(N^{(1)}_u\big)^{\alpha}}\geq e^{-t_{i,m}Z^{(1)}_{\gamma_{i,m}}-t'_{i,m}Z^{(1)}_{\gamma'_{i,m}}},
\end{align*}
that is
\begin{align*}
    \frac{\E\bigg[1\land e^{-t_{i,m}Z^{(1)}_{\gamma_{i,m}}-t'_{i,m}Z^{(1)}_{\gamma'_{i,m}}+r_{i,m}\sum_{|u|=\gamma'_{i,m}}\big(N^{(1)}_u\big)^{\alpha}}\Big|Z^{(1)}_{\gamma_{i,m}}>0\bigg]}{\E\bigg[e^{-t_{i,m}Z^{(1)}_{\gamma_{i,m}}-t'_{i,m}Z^{(1)}_{\gamma'_{i,m}}}\Big|Z^{(1)}_{\gamma_{i,m}}>0\bigg]}\geq 1.
\end{align*}
Taking $\varepsilon\to 0$ yields \eqref{RatioCondi}. We deduce from \eqref{RatioCondi} that
\begin{align*}
    \lim_{m\to\infty}\Bigg(&\E\bigg[1\land e^{-t_{i,m}Z^{(1)}_{\gamma_{i,m}}-t'_{i,m}Z^{(1)}_{\gamma'_{i,m}}+r_{i,m}\sum_{|u|=\gamma'_{i,m}}\big(N^{(1)}_u\big)^{\alpha}}\Big|Z^{(1)}_{\gamma_{i,m}}>0\bigg] \\ & -\E\bigg[e^{-t_{i,m}Z^{(1)}_{\gamma_{i,m}}-t'_{i,m}Z^{(1)}_{\gamma'_{i,m}}}\Big|Z^{(1)}_{\gamma_{i,m}}>0\bigg]\Bigg)=0,
\end{align*}
uniformly in $i\geq 1$ and since $\lim_{m\to\infty}\sup_{i\geq 1}p_{\gamma'_{i,m}}=0$, $\gamma_{i,m}\leq\gamma'_{i,m}$ and $\lim_{m\to\infty}p_{\gamma_{i,m}}/z_{i,m}=\mu$ uniformly in $i\geq 1$, we have
\begin{align*}
    & \lim_{m\to\infty}\left(\frac{\E\Big[1\land e^{-t_{i,m}Z^{(1)}_{\gamma_{i,m}}-t'_{i,m}Z^{(1)}_{\gamma'_{i,m}}+r_{i,m}\sum_{|u|=\gamma'_{i,m}}\big(N^{(1)}_u\big)^{\alpha}}\Big]}{\E\Big[e^{-t_{i,m}Z^{(1)}_{\gamma_{i,m}}-t'_{i,m}Z^{(1)}_{\gamma'_{i,m}}}\Big]}\right)^{1/z_{i,m}} \\ & =\lim_{m\to\infty}\left(\frac{1-p_{\gamma_{i,m}}\times\Big(1-\E\big[1\land e^{-t_{i,m}Z^{(1)}_{\gamma_{i,m}}-t'_{i,m}Z^{(1)}_{\gamma'_{i,m}}+r_{i,m}\sum_{|u|=\gamma'_{i,m}}\big(N^{(1)}_u\big)^{\alpha}}\big|Z^{(1)}_{\gamma_{i,m}}>0\big]\Big)}{1-p_{\gamma_{i,m}}\times\Big(1-\E\big[e^{-t_{i,m}Z^{(1)}_{\gamma_{i,m}}-t'_{i,m}Z^{(1)}_{\gamma'_{i,m}}}\big|Z^{(1)}_{\gamma_{i,m}}>0\big]\Big)}\right)^{1/z_{i,m}}=1,
\end{align*}
uniformly in $i\geq 1$, thus giving \eqref{LemmConv1}.
\end{proof}

\noindent Recall the definition of $\phi^{(\kappa)}(\cdot)$ in \eqref{PhiKappa}. We now make a link between the probability generating function of the multi-type additive martingale and $\phi^{(\kappa)}(\cdot)$.

\begin{lemm}\label{LemmaProp1}
For any integer $i\geq 1$, let $(\gamma_{i,m})_{m\geq 1}$ be a sequence of positive integers such that $0\leq\gamma^-_m\leq\gamma_{i,m}\leq\gamma^+_m$ where $\gamma^-_m/m\to 1$, $\gamma^+_m/m\to 1$ as $m\to\infty$. For any $\lambda\geq 0$
\begin{align}\label{ConvUnifEdge}
    \lim_{m\to\infty}\Big(G_{\gamma_{i,m}}\big(e^{-\lambda p_m}\big)\Big)^{1/p_m}=e^{-\big(1-\phi^{(\kappa)}(\lambda)\big)},
\end{align}
uniformly in $i\geq 1$.
\end{lemm}
\begin{proof}
The case $\lambda=0$ is trivial. Assume that $\lambda>0$. We follow an idea developed in the proof of Theorem 1 in \cite{Slack1968}. Since $e^{-\lambda p_m}\to1$ as $m\to\infty$, there exists a sequence $(k_m)_{m\geq 1}$ of integers (depending on $\lambda$) such that $k_m\to\infty$ as $m\to\infty$ and for any $m\geq 1$
\begin{align}\label{Encadrement}
    G_{k_m}(0)\leq e^{-\lambda p_m}<G_{k_m+1}(0).
\end{align}
Note that we necessarily have $m/k_m\to\lambda^{\kappa\land 2-1}$ as $m\to\infty$ for all $\kappa>1$. Indeed, noticing that $G_k(0)=1-p_k$, we have by Theorem \ref{Theorem1} that $\ell^{1/(\kappa\land 2-1)}(1-G_{\ell}(0))$ goes to $((\kappa\land 2-1)C_{\infty}\bm{\mathrm{C}}_{\kappa})^{-1/(\kappa\land 2-1)}$ as $\ell\to\infty$ when $\kappa\not=2$ and $(\ell\log \ell)(1-G_{\ell}(0))$ goes to $(C_{\infty}\bm{\mathrm{C}}_{2})^{-1}$ when $\kappa=2$ as $\ell\to\infty$. Moreover, by definition, $(1-e^{-\lambda p_m})/(1-G_m(0))\to\lambda$ as $m\to\infty$. Hence, by \eqref{Encadrement}, $\liminf_{m\to\infty}(1-G_{k_m})/(1-G_m(0))\geq\lambda$ and $\limsup_{m\to\infty}(1-G_{k_m+1})/(1-G_m(0))\leq\lambda$ thus giving $m/k_m\to\lambda^{\kappa\land 2-1}$ as $m\to\infty$ for all $\kappa>1$. \\
Let us now deal with the upper bound in \eqref{ConvUnifEdge}. We have
\begin{align*}
    G_{\gamma_{i,m}}\big(e^{-\lambda p_m}\big)\leq G_{\gamma_{i,m}}\big(G_{k_m+1}(0)\big)\leq G_{\gamma_{i,m}+k_m+1}(0)\leq G_{\gamma^+_{m}+k_m+1}(0)=1-p_{\gamma^+_m+k_m+1},
\end{align*}
where we have used \eqref{Encadrement} and the fact that $G_k$ is non-decreasing for the first inequality, Lemma \ref{FoncGen} (lower bound) saying that $G_{k}(G_{\ell}(s))\leq G_{k+\ell}(s)$ for the second one and the fact that the sequence $(G_k(0))_{k\geq 1}$ is non-decreasing for the last inequality. Using again Theorem \ref{Theorem1}, we get that the sequence $(p_{\gamma^+_m+k_m+1}/p_m)_m$ converges to $(1+\lambda^{-(\kappa\land 2-1)})^{-1/(\kappa\land 2-1)}=1-\phi^{(\kappa)}(\lambda)$ as $m\to\infty$. Therefore, uniformly in $i\geq 1$
\begin{align*}
    \limsup_{n\to\infty}\Big(G_{\gamma_{i,m}}\big(e^{-\lambda p_m}\big)\Big)^{1/p_m}\leq e^{-\big(1-\phi^{(\kappa)}(\lambda)\big)}.
\end{align*}
We now deal with the lower bound in \eqref{ConvUnifEdge}. Note that $$\frac{1}{p_{k_m}}\Eb\big[(p^{\mathcal{E}}_{k_m})^{\alpha}\big]\leq(p_{k_m})^{\alpha-1}\sup_{\ell\geq 0}\Eb\big[(p^{\mathcal{E}}_{\ell}/p_{\ell})^{\alpha}\big]\to 0$$ as $m\to\infty$ for any $\alpha\in(1,\kappa\land 2)$ by Remark \ref{BoundedInLp} (we recall that $k_m\to\infty$). Now, if $t_{i,m}=p_{k_m}$ and $s_{i,m}=-\log G_{k_m}(0)=-\log(1-p_{k_m})$ for all $i\geq 1$, then by Theorem \ref{Theorem1}, since $m/k_m\to\lambda^{\kappa\land 2-1}$, we have $\sup_{i\geq 1}|t_{i,m}/p_{\gamma_{i,m}}-\lambda|\to0$ and $\sup_{i\geq 1}|s_{i,m}/p_{\gamma_{i,m}}-\lambda|\to0$ as $m\to\infty$. Hence, by Lemma \ref{NewEquivFoncGen} \eqref{LemmConv2} with $r_{i,m}=\Eb[(p^{\mathcal{E}}_{k_m})^{\alpha}]$ and $z_{i,m}=p_m$, we have
\begin{align*}
    G_{\gamma_{i,m}}\big(G_{k_m}(0)\big)=\big(1+o(1)\big)^{p_m}\E\left[1\land e^{-p_{k_{m}}Z^{(1)}_{\gamma_{i,m}}+\sum_{|u|=\gamma_{i,m}}\big(N^{(1)}_u\big)^{\alpha}\Eb\big[\big(p^{\mathcal{E}}_{k_m}\big)^{\alpha}\big]}\right], 
\end{align*}
as $m\to\infty$, uniformly in $i\geq 1$. Besides, we know from Lemma \ref{FoncGen} with $s=0$ that
\begin{align*}
    \E\left[1\land e^{-p_{k_{m}}Z^{(1)}_{\gamma_{i,m}}+\sum_{|u|=\gamma_{i,m}}\big(N^{(1)}_u\big)^{\alpha}\Eb\big[\big(p^{\mathcal{E}}_{k_m}\big)^{\alpha}\big]}\right]\geq 1-p_{\gamma_{i,m}+k_m}.
\end{align*}
Note that $1-p_{\gamma_{i,m}+k_m}\geq 1-p_{\gamma^-_m+k_m}$ and thanks to \eqref{Encadrement}, $G_{\gamma_{i,m}}\big(G_{k_m}(0)\big)\leq G_{\gamma_{i,m}}(e^{-\lambda p_m})$ thus giving
\begin{align*}
    \liminf_{m\to\infty}\Big(G_{\gamma_{i,m}}\big(e^{-\lambda p_m}\big)\Big)^{1/p_m}\geq\lim_{m\to\infty}\big(1-p_{\gamma^-_m+k_m}\big)^{1/p_m}=e^{-\big(1-\phi^{(\kappa)}(\lambda)\big)},
\end{align*}
uniformly in $i\geq 1$ and this gives \eqref{ConvUnifEdge}, where we have used that the sequence $(p_{\gamma^-_m+k_m}/p_m)_m$ converges to $(1+\lambda^{-(\kappa\land 2-1)})^{-1/(\kappa\land 2-1)}=1-\phi^{(\kappa)}(\lambda)$ and the proof of the lemma is completed.
\end{proof}
\noindent The next result is an analogue of Lemma \ref{LemmaProp1} for generation negligible with respect to $m$:
\begin{lemm}\label{LemmaProp2}
For any $i\geq 1$, let $(a_{i,m})_{m\geq 1}$ be a sequence of positive integers such that $0\leq a^-_m\leq a_{i,m}\leq a^+_m$ where $a^+_m=o(m)$ as $m\to\infty$, we have
\begin{align}\label{ConvUnifEdgeSmall}
    \lim_{m\to\infty}\Big(G_{a_{i,n}}\big(e^{-\lambda p_m}\big)\Big)^{1/p_m}=e^{-\lambda},
\end{align}
uniformly in $i\geq 1$.
\end{lemm}

\begin{proof}
Indeed, we still have (see the proof of Lemma \ref{LemmaProp1})
\begin{align*}
    \big(1+o(1)\big)^{p_m}\big(1-p_{a^-_m+k_m}\big)\leq G_{a_{i,n}}\big(e^{-\lambda p_m}\big)\leq 1-p_{a^+_m+k_m},
\end{align*}
as $m\to\infty$ and uniformly in $i\geq 1$. Finally, recalling that $m/k_m\to\lambda^{\kappa\land 2-1}$ and since $a^-_m\leq a^+_m=o(m)$, we have $\lim_{m\to\infty}p_{a^-_m+k_m}/p_m=\lim_{m\to\infty}p_{a^+_m+k_m}/p_m=\lim_{m\to\infty}(m/k_m)^{1/(\kappa\land 2-1)}=\lambda$ thus giving \eqref{ConvUnifEdgeSmall}.
\end{proof}

\noindent We are now ready to prove our main statement.

\begin{proof}[Proof of Proposition \ref{UnifLawMultyMart}]

For any integer $i\geq 1$, $(\gamma_{i,m})_{m\geq 1}$ and $(\rho_{i,m})_{m\geq 1}$ are two sequences of positive integers such that $0\leq\gamma^-_m\leq\gamma_{i,m}\leq\gamma^+_m$ and $0\leq\rho^-_m\leq\rho_{i,m}\leq\rho^+_m$ where $\gamma^-_m/m\to 1$, $\gamma^+_m/m\to 1$ and $\rho^+_m/m\to 0$ as $m\to\infty$. We want to prove that for any $\lambda_1,\lambda_2\geq 0$, uniformly in $i\geq 1$
\begin{align*}
    \lim_{m\to\infty}\Big(\E\Big[e^{-\lambda_1 p_{m}L^{(1)}_{\gamma_{i,m}}}e^{-\lambda_2p_mL^{(1)}_{\gamma_{i,m}+\rho_{i,m}}}\Big]\Big)^{1/p_m}=e^{-\big(1-\phi^{(\kappa)}(2\lambda_1+2\lambda_2)\big)}.
\end{align*}
One can only focus on the case $\lambda_1,\lambda_2>0$. Recall that $L^{(1)}_k=Z^{(1)}_k+Z^{(1)}_{k+1}$ so thanks to the branching property (see Fact \ref{FactGWMulti})
\begin{align*}
    \E\Big[e^{-\lambda_1p_mL^{(1)}_{\gamma_{i,m}}}e^{-\lambda_2p_mL^{(1)}_{\gamma_{i,m}+\rho_{i,m}}}\Big]=\E\Big[e^{-\lambda_1p_mL^{(1)}_{\gamma_{i,m}}}e^{-\lambda_2p_mZ^{(1)}_{\gamma_{i,m}+\rho_{i,m}}}\prod_{u\in\mathcal{R}^{(1)};|u|=\gamma_{i,m}+\rho_{i,m}}G_{1}\Big(e^{-\lambda_2p_m},N_u^{(1)}\Big)\Big].
\end{align*}
For the lower bound, by Lemma \ref{FoncGen}, we have $G_{\ell}(s,p)\geq (G_{\ell}(s))^p$ so
\begin{align*}
    \E\left[e^{-\lambda_1p_mL^{(1)}_{\gamma_{i,m}}}e^{-\lambda_2p_mL^{(1)}_{\gamma_{i,m}+\rho_{i,m}}}\right]\geq\E\Big[e^{-\lambda_1p_mL^{(1)}_{\gamma_{i,m}}}\Big(e^{-\lambda_2p_m}G_1\big(e^{-\lambda_2p_m}\big)\Big)^{Z^{(1)}_{\gamma_{i,m}+\rho_{i,m}}}\Big].
\end{align*}
Again, thanks to the branching property and the fact that $G_{\ell}(s,p)\geq (G_{\ell}(s))^p$, we have
\begin{align*}
    &\E\Big[e^{-\lambda_1p_mL^{(1)}_{\gamma_{i,m}}}\Big(e^{-\lambda_2p_m}G_1\big(e^{-\lambda_2p_m}\big)\Big)^{Z^{(1)}_{\gamma_{i,m}+\rho_{i,m}}}\Big] \\ & =\E\Big[e^{-\lambda_1p_mL^{(1)}_{\gamma_{i,m}}}\prod_{u\in\mathcal{R}^{(1)};|u|=\gamma_{i,m}+1}G_{\rho_{i,m}-1}\Big(e^{-\lambda_2p_m}G_1\big(e^{-\lambda_2p_m}\big),N_u^{(1)}\Big)\Big]\geq\E\left[e^{-\lambda_1p_mZ^{(1)}_{\gamma_{i,m}}}\big(\mathfrak{g}_{i,m}\big)^{Z^{(1)}_{\gamma_{i,m}+1}}\right],
\end{align*}
where 
\begin{align*}
    \mathfrak{g}_{i,m}:=e^{-\lambda_1p_m}G_{\rho_{i,m}-1}\Big(e^{-\lambda_2p_m}G_1\big(e^{-\lambda_2p_m}\big)\Big).
\end{align*}
Doing this one more time and we finally obtain
\begin{align*}
    \E\Big[e^{-\lambda_1p_mL^{(1)}_{\gamma_{i,m}}}e^{-\lambda_2p_mL^{(1)}_{\gamma_{i,m}+\rho_{i,m}}}\Big]\geq G_{\gamma_{i,m}}\left(e^{-\lambda_1p_m}G_1\left(\mathfrak{g}_{i,m}\right)\right).
\end{align*}
By Lemma \ref{LemmaProp2}, we have, for any $\lambda\geq 0$ as $m\to\infty$
\begin{align}\label{EstiG1}
    \lim_{m\to\infty}\Big(G_1\big(e^{-\lambda p_m}\big)\Big)^{1/p_m}=e^{-\lambda}.
\end{align}
In particular, $\lim_{m\to\infty}(e^{-\lambda_2p_m}G_1(e^{-\lambda_2p_m}))^{1/p_m}=e^{-2\lambda_2}$ and recalling that $\rho^-_m\leq\rho_{i,m}\leq\rho^+_m=o(m)$, Lemma \ref{LemmaProp2} yields $\lim_{m\to\infty}(\mathfrak{g}_{i,m})^{1/p_m}=e^{-\lambda_1-2\lambda_2}$. Since $G_1$ is non-decreasing, we obtain from \eqref{EstiG1} that $$\lim_{m\to\infty}(G_1(\mathfrak{g}_{i,m}))^{1/p_m}=e^{-\lambda_1-2\lambda_2},$$ uniformly in $i\geq 1$. Then, we get that $\lim_{m\to\infty}(e^{-\lambda_1p_m}G_1(\mathfrak{g}_{i,m}))^{1/p_m}=e^{-2\lambda_1-2\lambda_2}$, uniformly in $i\geq 1$. Hence, recalling that $\gamma^-_m\leq\gamma_{i,m}\leq\gamma^+_m$ with $\lim_{m\to\infty}\gamma^-_m/m=\lim_{m\to\infty}\gamma^+_m/m=1$ as $m\to\infty$ and that the sequence $(p_k)_{k\geq 1}$ is non-increasing, we have, by Lemma \ref{NewEquivFoncGen} with $t_{i,m}=(2\lambda
_1+2\lambda_2)p_m$ for all $i\geq 1$, $s_{i,m}=\lambda_1p_m-\log G_1(\mathfrak{g}_{i,m})$ and $z_{i,m}=p_m$, that 
\begin{align*}
    \E\Big[e^{-\lambda_1p_mL^{(1)}_{\gamma_{i,m}}}e^{-\lambda_2p_mL^{(1)}_{\gamma_{i,m}+\rho_{i,m}}}\Big]\geq \big(1+o(1)\big)^{p_m}G_{\gamma_{i,m}}\Big(e^{-2(\lambda_1+\lambda_2)p_m}\Big),
\end{align*}
as $m\to\infty$, uniformly in $i\geq 1$. Therefore, Lemma \ref{LemmaProp1} yields
\begin{align*}
    \liminf_{m\to\infty}\E\Big[e^{-\lambda_1p_mL^{(1)}_{\gamma_{i,m}}}e^{-\lambda_2p_mL^{(1)}_{\gamma_{i,m}+\rho_{i,m}}}\Big]^{1/p_m}\geq e^{-\big(1-\phi^{(\kappa)}(2\lambda_1+2\lambda_2)\big)},
\end{align*}
uniformly in $i\geq 1$. For the upper bound, we have, thanks to the branching property (see Fact \ref{FactGWMulti}) and Lemma \ref{FoncGen} (upper bound), that for any $\alpha\in(1,\kappa\land 2)$, $\E[e^{-\lambda_1p_mL^{(1)}_{\gamma_{i,m}}}e^{-\lambda_2p_mL^{(1)}_{\gamma_{i,m}+\rho_{i,m}}}]$ is smaller than
\begin{align*}
    &\E\Bigg[e^{-\lambda_1p_mL^{(1)}_{\gamma_{i,m}}-\lambda_2p_mZ^{(1)}_{\gamma_{i,m}+\rho_{i,m}}}\Bigg(1\land e^{-\big(1-G_1(e^{-\lambda_2p_m})\big)Z^{(1)}_{\gamma_{i,m}+\rho_{i,m}}+\sum_{|u|=\gamma_{i,m}+\rho_{i,m}}\big(N^{(1)}_u\big)^{\alpha}\Eb\big[\big(1-G^{\mathcal{E}}_1(e^{-\lambda_2p_m})\big)^{\alpha}\big]}\Bigg)\Bigg] \\ & \leq\E\Bigg[1\land e^{-\lambda_1p_mL^{(1)}_{\gamma_{i,m}}}e^{-\big(\lambda_2p_m+1-G_1(e^{-\lambda_2p_m})\big)Z^{(1)}_{\gamma_{i,m}+\rho_{i,m}}+\sum_{|u|=\gamma_{i,m}+\rho_{i,m}}\big(N^{(1)}_u\big)^{\alpha}\Eb\big[\big(1-G^{\mathcal{E}}_1(e^{-\lambda_2p_m})\big)^{\alpha}\big]}\Bigg].
\end{align*}
Note that $$\big(1-G^{\mathcal{E}}_1(e^{-\lambda_2p_m})\big)^{\alpha}=\big(\E^{\mathcal{E}}[1-e^{-\lambda_2p_mZ^{(1)}_1}]\big)^{\alpha}\leq\big(\E^{\mathcal{E}}[\lambda_2p_mZ^{(1)}_1]\big)^{\alpha}=(\lambda_2p_m)^{\alpha}(W_1)^{\alpha}.$$ Recall that under the Assumptions \ref{Assumption1} and \ref{Assumption2}, we have $\sup_{\ell\geq 0}\Eb[(W_{\ell})^{\alpha}]<\infty$ so
\begin{align*}
    \lim_{m\to\infty}\sup_{i\geq 1}\frac{1}{p_{\gamma_{i,m}}}\Eb\big[\big(1-G^{\mathcal{E}}_1(e^{-\lambda_2p_m})\big)^{\alpha}\big]\leq\lim_{m\to\infty}\frac{(\lambda_1p_m)^{\alpha}}{p_{\gamma^+_m}}\Eb\big[(W_1)^{\alpha}\big]=0,
\end{align*}
since $\alpha>1$, $p_m/p_{\gamma^+_m}\to1$ and $p_{\gamma^+_m}\to 0$ as $m\to\infty$. Also note that $\lim_{m\to\infty} (1-G_1(e^{-\lambda_2p_m}))/p_m=\lambda_2$ by Lemma \ref{LemmaProp2}. In particular
\begin{align*}
    \sup_{m,i\geq 1}\frac{\lambda_2p_m+1-G_1(e^{-\lambda_2p_m})}{p_{\gamma_{i,m}+\rho_{i,m}}}\leq\sup_{m\geq 1}\frac{\lambda_2p_m+1-G_1(e^{-\lambda_2p_m})}{p_{\gamma^+_m+\rho^+_m}}<\infty.
\end{align*}
Hence, thanks to Lemma \ref{NewEquivFoncGen} \eqref{LemmConv1} with $\gamma'_{i,m}=\gamma_{i,m}+\rho_{i,m}$, $t_{i,m}=\lambda_1p_m$, $t'_{i,m}=\lambda_2p_m+1-G_1(e^{-\lambda_2p_m})$, $r_{i,m}=\Eb[\big(1-G^{\mathcal{E}}_1(e^{-\lambda_2p_m})\big)^{\alpha}]$ and $z_{i,m}=p_m$, we have
\begin{align*}
    &\E\Bigg[1\land e^{-\lambda_1p_mL^{(1)}_{\gamma_{i,m}}}e^{-\big(\lambda_2p_m+1-G_1(e^{-\lambda_2p_m})\big)Z^{(1)}_{\gamma_{i,m}+\rho_{i,m}}+\sum_{|u|=\gamma_{i,m}+\rho_{i,m}}\big(N^{(1)}_u\big)^{\alpha}\Eb\big[\big(1-G^{\mathcal{E}}_1(e^{-\lambda_2p_m})\big)^{\alpha}\big]}\Bigg] \\ & =\big(1+o(1)\big)^{p_m}\E\Bigg[e^{-\lambda_1p_mL^{(1)}_{\gamma_{i,m}}-\big(\lambda_2p_m+1-G_1(e^{-\lambda_2p_m})\big)Z^{(1)}_{\gamma_{i,m}+\rho_{i,m}}}\Bigg],
\end{align*}
as $m\to\infty$, uniformly in $i\geq 1$. Repeating this procedure twice, one can prove that
\begin{align*}
    \E\Bigg[e^{-\lambda_1p_mL^{(1)}_{\gamma_{i,m}}-\big(\lambda_2p_m+1-G_1(e^{-\lambda_2p_m})\big)Z^{(1)}_{\gamma_{i,m}+\rho_{i,m}}}\Bigg]\leq\big(1+o(1)\big)^{p_m}G_{\gamma_{i,m}}\big(e^{-\mathfrak{b}_{i,m}}\big),
\end{align*}
as $m\to\infty$ and uniformly in $m\geq 1$ where
\begin{align*}
    \mathfrak{b}_{i,m}:=\lambda_1p_m+1-G_1\big(e^{-\mathfrak{a}_{i,m}}\big)\;\textrm{ and }\;\mathfrak{a}_{i,m}=\lambda_1p_m+1-G_{\rho_{i,m}-1}\big(e^{-(\lambda_2p_m+1-G_1(e^{-\lambda_2p_m}))}\big).
\end{align*}
Let us prove that $\mathfrak{b}_{i,m}/p_m\to 2\lambda_1+2\lambda_2$ as $m\to\infty$, uniformly in $i\geq 1$. As already proved above, we have $\lim_{m\to\infty}(\lambda_2p_m+1-G_1(e^{-\lambda_2p_m}))/p_m\to 2\lambda_2$ as $m\to\infty$ and since $\rho^-\leq\rho_{i,m}\leq\rho_m^+=o(m)$, Lemma \ref{LemmaProp2} yields $(1-G_{\rho_{i,m}-1}(e^{-(\lambda_2p_m+1-G_1(e^{-\lambda_2p_m}))}))/p_m\to 2\lambda_2$ as $m\to\infty$, uniformly in $i\geq 1$, that is $\mathfrak{a}_{i,m}/p_m\to\lambda_1+2\lambda_2$ as $m\to\infty$, uniformly in $i\geq 1$. Similarly, one can see that $1-G_1(e^{-\mathfrak{a}_{i,m}})\to\lambda_1+2\lambda_2$ as $m\to\infty$, uniformly in $i\geq 1$ and this gives the convergence of $\mathfrak{b}_{i,m}/p_m$. Now, we can use that $\mathfrak{b}_{i,m}/((2\lambda_1+2\lambda_2)p_m)\to1$ as $m\to\infty$, uniformly in $i\geq 1$, together with Lemma \ref{NewEquivFoncGen} \eqref{LemmConv2} with $t_{i,m}=\mathfrak{b}_{i,m}$, $s_{i,m}=(2\lambda_1+2\lambda_2)p_m$, $r_{i,m}=0$ and $z_{i,m}=p_m$ to get
\begin{align*}
    \Big(G_{\gamma_{i,m}}\big(e^{-\mathfrak{b}_{i,m}}\big)\Big)^{1/p_m}=\big(1+o(1)\big)\Big(G_{\gamma_{i,m}}\big(e^{-(2\lambda_1+2\lambda_2)p_m}\big)\Big)^{1/p_m},
\end{align*}
as $m\to\infty$, uniformly in $i\geq 1$ and Lemma \ref{LemmaProp1} yields 
\begin{align*}
    \lim_{m\to\infty}\Big(G_{\gamma_{i,m}}\big(e^{-\mathfrak{b}_{i,m}}\big)\Big)^{1/p_m}=e^{-\big(1-\phi^{(\kappa)}(2\lambda_1+2\lambda_2)\big)},
\end{align*}
uniformly in $i\geq 1$. Hence
\begin{align*}
    \limsup_{m\to\infty}\E\Big[e^{-\lambda_1p_mL^{(1)}_{\gamma_{i,m}}}e^{-\lambda_2p_mL^{(1)}_{\gamma_{i,m}+\rho_{i,m}}}\Big]^{1/p_m}\leq e^{-\big(1-\phi^{(\kappa)}(2\lambda_1+2\lambda_2)\big)},
\end{align*}
uniformly in $i\geq 1$, thus giving the upper bound we wanted and the proof of the proposition is completed.
\end{proof}

\noindent We now use Proposition \ref{UnifLawMultyMart} in order to prove our second theorem.

\subsection{Proof of Theorem \ref{Theorem3}}\label{ProofTheorem3}

We only deal with the case $\kappa\in(1,2)$, the proof is the same for $\kappa\geq 2$. We want to prove that for any $a>0$, in $\Pb^*$-law
\begin{align*}
    \frac{1}{n}L^{(n)}_{\lfloor an^{\kappa-1}\rfloor}\underset{n\to\infty}{\longrightarrow} 2W_{\infty}\mathcal{Y}^{(\kappa)}_{a/(W_{\infty})^{\kappa-1}}\;,
\end{align*}
where we recall the definition of $\mathcal{Y}^{(\kappa)}$ in \eqref{LaplaceCSBP}. The proof is divided into four steps. In the first step, we prove that under the annealed probability $\P$, the local time can be approximated by the reduced local time, in the sense that in $\P$-probability, $\frac{1}{n}L^{(n)}_{\lfloor an^{\kappa-1}\rfloor}-\frac{1}{n}\bm{L}^{(n)}_{\lfloor an^{\kappa-1}\rfloor}(\Bar{B}^{(n)})$ goes to $0$ as $n\to\infty$. As already mentioned before, the reduced range is more convenient to work with under the annealed probability $\P$. In the second step, we show that the sequence of random processes $((\frac{1}{n}\bm{L}^{(n)}_{\lfloor an^{\kappa-1}\rfloor}(\tn);\; t\geq 0))_{n\geq 1}$ converges in law under $\P$ to $(2t\mathcal{Y}^{(\kappa)}_{a/t^{\kappa-1}};\; t\geq 0)$ thanks to Proposition \ref{UnifLawMultyMart} and standard tightness arguments. Using the same idea as in proof of Theorem 1.1 in \cite{AidRap}, we prove in the third step that the latter convergence holds under the quenched probability $\P^{\mathcal{E}}$ for almost every environment $\mathcal{E}$. Finally, we take $t=\frac{\CardRoots}{n}$ in the last step and conclude using \eqref{ConvCardRoot}, stating that $\frac{\CardRoots}{n}$ converges to $W_{\infty}$ in $\P^*$-probability as $n\to\infty$.

\vspace{0.2cm}

\noindent \textbf{Step 1}: we can restrict ourselves to the convergence of the reduced local time.

\vspace{0.1cm}

\noindent Let us first prove that for any $\varepsilon>0$ and any $a>0$
\begin{align}\label{DiffLoctime}
    \lim_{n\to\infty}\P\Big(\Big| L^{(n)}_{\lfloor an^{\kappa-1}\rfloor}-\bm{L}^{(n)}_{\lfloor an^{\kappa-1}\rfloor}\big(\Bar{B}^{(n)}\big)\Big|>n\varepsilon\Big)=0,
\end{align}
where we recall that $\bm{L}^{(n)}_{k}(q)=\sum_{i=1}^{q}\bm{L}^{(n,i)}_k$ with $\bm{L}^{(n,i)}_{k}=\bm{Z}^{(n,i)}_{k}+\bm{Z}^{(n,i)}_{k+1}$ and $\bm{Z}^{(n,i)}_k=\sum_{x\in\mathbfcal{R}^{(n)}_i}\un_{\{|x|_i=k\}}\bm{t}^{(n)}_x$.
Thanks to \eqref{ApproxReduced}, it is enough to show
\begin{align}\label{RestrictionToProcess}
    \lim_{n\to\infty}\P\left(\Big|\sum_{i=1}^{\CardRoots}\bm{L}^{(n,i)}_{\lfloor an^{\kappa-1}\rfloor-|e_i|}-\bm{L}^{(n)}_{\lfloor an^{\kappa-1}\rfloor}\big(\Bar{B}^{(n)}\big)\Big|>n\varepsilon\right)=0.
\end{align}
For that, we prove
\begin{align}\label{JointCVinLaw}
    \E\Big[e^{-\frac{\lambda_1}{n}\sum_{i=1}^{\CardRoots}\bm{L}^{(n,i)}_{\lfloor an^{\kappa-1}\rfloor-|e_i|}}e^{-\frac{\lambda_2}{n}\bm{L}^{(n)}_{\lfloor an^{\kappa-1}\rfloor}(\Bar{B}^{(n)})}\Big]\underset{n\to\infty}{\longrightarrow}\E\Big[e^{-2(\lambda_1+\lambda_2)W_{\infty}\mathcal{Y}^{(\kappa)}_{a/W_{\infty}}}\Big],
\end{align}
for any $\lambda_1,\lambda_2\geq 0$. Let $q_n:=p_{\lfloor an^{\kappa-1}\rfloor}$ and let $\mathcal{U}:=\bigcup_{j\geq 0}(\N)^j$, the set of finite $\N$-valued sequences, with the convention that $\N^0$ only contains the sequence with length $0$. 
For any $b\geq 0$, introduce $$\bm{\mathrm{C}}_{\kappa,b}:=(b(\kappa-1)C_{\infty}\bm{\mathrm{C}}_{\kappa})^{-1/(\kappa-1)},$$ and recall the definition of $\bm{\mathrm{C}}_{\kappa}$ in \eqref{DefCkappa}. By definition of the reduced range, we have
\begin{align*}
    \Big|\E\Big[e^{-\lambda_1q_n\sum_{i=1}^{\CardRoots}\bm{L}^{(n,i)}_{\lfloor an^{\kappa-1}\rfloor-|e_i|}}e^{-\lambda_2q_n\bm{L}^{(n)}_{\lfloor an^{\kappa-1}\rfloor}(\Bar{B}^{(n)})}\Big]-\Eb\Big[e^{-\bm{\mathrm{C}}_{\kappa,a}W_{\infty}\big(1-\phi^{(\kappa)}(2\lambda_1+2\lambda_2)\big)}\Big]\Big|
\end{align*}
is smaller than
\begin{align}\label{JointCVLaw}
    &\underset{\lfloor(\log n)^2\rfloor\leq|u(i)|\leq\lfloor(\log n)^3\rfloor}{\sum_{U=\{u(i);\;i\geq 1\}\subset\mathcal{U}}}\P\big(\Bar{\mathcal{B}}^{(n)}=U\big)\nonumber \\[0.7em] & \times\E\Bigg[\Bigg|\prod_{i=1}^{\Bar{B}^{(n)}}\E\left[e^{-\lambda_1q_nL^{(1)}_{\kappa_{i,n}}}e^{-\lambda_2q_nL^{(1)}_{\lfloor an^{\kappa-1}\rfloor}}\right]-\Eb\Big[e^{-\bm{\mathrm{C}}_{\kappa,a}W_{\infty}\big(1-\phi^{(\kappa)}(2\lambda_1+2\lambda_2)\big)}\Big]\Bigg|\Bigg]+o(1),
\end{align}
where $\kappa_{i,n}:=\lfloor an^{\kappa-1}\rfloor-|u(i)|$ and we have used that $$\lim_{n\to\infty}\P^{\mathcal{E}}\big(\forall\; x\in\Bar{\mathcal{B}}^{(n)},\; \lfloor(\log n)^2\rfloor\leq|x|\leq\lfloor(\log n)^3\rfloor,\; \Bar{\mathcal{B}}^{(n)}\subset\mathcal{B}^{(n)}\big)=1$$ $\Pb$-almost surely. 
Now, by Proposition \ref{UnifLawMultyMart}
\begin{align*}
    \lim_{n\to\infty}\E\Big[e^{-\lambda_1q_nL^{(1)}_{\kappa_{i,n}}}e^{-\lambda_2q_nL^{(1)}_{\lfloor an^{\kappa-1}\rfloor}}\Big]^{1/q_n}=e^{-\big(1-\phi^{(\kappa)}(2\lambda_1+2\lambda_2)\big)},
\end{align*}
uniformly in $i\geq 1$. By \eqref{ConvCardRoot}, we have $\Bar{B}^{(n)}/n\to W_{\infty}$ in $\P$-probability so Theorem \ref{Theorem1} gives $(q_n\Bar{B}^{(n)})$ converges to $\bm{\mathrm{C}}_{\kappa,a}W_{\infty}$. Hence, \eqref{JointCVLaw} yields
\begin{align*}
    \lim_{n\to\infty}\E\Big[e^{-\lambda_1q_n\sum_{i=1}^{\CardRoots}\bm{L}^{(n,i)}_{\lfloor an^{\kappa-1}\rfloor-|e_i|}}e^{-\lambda_2q_n\bm{L}^{(n)}_{\lfloor an^{\kappa-1}\rfloor}(\Bar{B}^{(n)})}\Big]=\Eb\Big[e^{-\bm{\mathrm{C}}_{\kappa,a}W_{\infty}\big(1-\phi^{(\kappa)}(2\lambda_1+2\lambda_2)\big)}\Big].
\end{align*}
To obtain \eqref{JointCVinLaw}, we are only left to check that 
\begin{align*}
    \E\Big[e^{-2(\lambda_1+\lambda_2)\bm{\mathrm{C}}_{\kappa,a}W_{\infty}\mathcal{Y}^{(\kappa)}_{a/W_{\infty}}}\Big]=\Eb\Big[e^{-\bm{\mathrm{C}}_{\kappa,a}W_{\infty}\big(1-\phi^{(\kappa)}(2\lambda_1+2\lambda_2)\big)}\Big].
\end{align*}
Indeed, by definition of $\mathcal{Y}^{(\kappa)}$ in \eqref{LaplaceCSBP}, we have, for any $\theta\geq 0$ and any $b\geq 0$
\begin{align*}
   \E^{\mathcal{E}}\Big[e^{-\theta\mathcal{Y}^{(\kappa)}_b}\Big]=e^{-\bm{\mathrm{C}}_{\kappa,b}\big(1-\phi^{(\kappa)}(\theta/\bm{\mathrm{C}}_{\kappa,b})\big)},
\end{align*}
so taking $b=a/(W_{\infty})^{\kappa-1}$ and $\theta=2(\lambda_1+\lambda_2)\bm{\mathrm{C}}_{\kappa,a}W_{\infty}$ gives $\bm{\mathrm{C}}_{\kappa,b}=\bm{\mathrm{C}}_{\kappa,a}W_{\infty}$ and $\phi^{(\kappa)}(\theta/\bm{\mathrm{C}}_{\kappa,b})=2(\lambda_1+\lambda_2)$ so the proof of \eqref{JointCVinLaw} is completed. We now deduce from \eqref{JointCVinLaw} that the sequence of random variables $(\frac{1}{n}\sum_{i=1}^{\CardRoots}\bm{L}^{(n,i)}_{\lfloor an^{\kappa-1}\rfloor-|e_i|}-\frac{1}{n}\bm{L}^{(n)}_{\lfloor an^{\kappa-1}\rfloor}(\Bar{B}^{(n)}))_{n\geq 1}$ converges in law under $\P$ to $0$, thus implying the convergence to $0$ in $\P$-probability (and also in $\P^*$-probability) and we finally get \eqref{RestrictionToProcess}. In view of \eqref{RestrictionToProcess}, we only have to prove that in $\Pb^*$-law
\begin{align}\label{ConvPlaw}
    \frac{1}{n}\bm{L}^{(n)}_{\lfloor an^{\kappa-1}\rfloor}\big(\Bar{B}^{(n)}\big)\underset{n\to\infty}{\longrightarrow}2W_{\infty}\mathcal{Y}^{(\kappa)}_{a/(W_{\infty})^{\kappa-1}}\;.
\end{align}
The next steps are devoted to the proof of \eqref{ConvPlaw}.

\vspace{0.2cm}

\noindent \textbf{Step 2}: convergence of the reduced local time under $\P$

\vspace{0.1cm}

\noindent This step is dedicated to the following: we show that under the annealed probability $\P$, in law for the Skorokhod topology on $D((0,\infty],\R)$
\begin{align}\label{ConvAnnealedProcess}
    \Big(\frac{1}{n}\bm{L}^{(n)}_{\lfloor an^{\kappa-1}\rfloor}(\tn);\; t\geq 0\Big)\underset{n\to\infty}{\longrightarrow}\big(2t\mathcal{Y}^{(\kappa)}_{a/t^{\kappa-1}};\; t\geq 0\big).
\end{align}
Recall that under $\P$, $\bm{L}^{(n)}_k(q)$ is a sum of $q$ independent copies of $L^{(1)}_k$. Therefore, the convergence of finite-dimensional marginals comes from \textbf{Step 1} with (by replacing $\Bar{B}^{(n)}$ with $\tn$ and then $W_{\infty}$ with $t$ and taking $\lambda_1=0$ in \eqref{JointCVinLaw}). Let us now check that, in law under $\P$, $2\mathcal{Y}^{(\kappa)}_{a}-2(1-\delta)\mathcal{Y}^{(\kappa)}_{a/(1-\delta)^{\kappa-1}}\to 0$ as $\delta\to0$. For any $\lambda_1,\lambda_2\geq 0$, we have, by \eqref{LaplaceCSBP}
\begin{align*}
    \E\Big[e^{-\lambda_1\mathcal{Y}^{(\kappa)}_{a}}e^{-\lambda_2(1-\delta)\mathcal{Y}^{(\kappa)}_{a/(1-\delta)^{\kappa-1}}}\Big]=\E\Big[e^{-C_{\delta,\lambda_1,\lambda_2}\mathcal{Y}^{(\kappa)}_a}\Big],
\end{align*}
where $C_{\delta,\lambda_1,\lambda_2}:=\lambda_1+\lambda_2(1+((1-\delta)^{1-\kappa}-1)(\lambda_1(1-\delta)/\bm{\mathrm{C}}_{\kappa,a})^{\kappa-1})^{-1/(\kappa-1)}$. Note that $\lim_{\delta\to0}C_{\delta,\lambda_1,\lambda_2}=\lambda_1+\lambda_2$ so
\begin{align*}
    \lim_{\delta\to 0}\E\Big[e^{-\lambda_1\mathcal{Y}^{(\kappa)}_{a}}e^{-\lambda_2(1-\delta)\mathcal{Y}^{(\kappa)}_{a/(1-\delta)^{\kappa-1}}}\Big]=\E\Big[e^{-(\lambda_1+\lambda_2)\mathcal{Y}^{(\kappa)}_a}\Big],
\end{align*}
and in particular, $2\mathcal{Y}^{(\kappa)}_{a}-2(1-\delta)\mathcal{Y}^{(\kappa)}_{a/(1-\delta)^{\kappa-1}}\to 0$ in law under $\P$ as $\delta\to\infty$. Finally, let $B>0$, $0<r<s<t$ and $n\geq 1$. We have
\begin{align*}
    \P\Big(\big|\bm{L}^{(n)}_{\lfloor an^{\kappa-1}\rfloor}(\sn)-\bm{L}^{(n)}_{\lfloor an^{\kappa-1}\rfloor}(\rn)\big|\land\big|\bm{L}^{(n)}_{\lfloor an^{\kappa-1}\rfloor}(\tn)-\bm{L}^{(n)}_{\lfloor an^{\kappa-1}\rfloor}(\sn)\big|\geq Bn\Big)\leq\frac{16}{B^2}(t-r)^2.
\end{align*}
Indeed
\begin{align*}
    &\P\Big(\big|\bm{L}^{(n)}_{\lfloor an^{\kappa-1}\rfloor}(\sn)-\bm{L}^{(n)}_{\lfloor an^{\kappa-1}\rfloor}(\rn)\big|\land\big|\bm{L}^{(n)}_{\lfloor an^{\kappa-1}\rfloor}(\tn)-\bm{L}^{(n)}_{\lfloor an^{\kappa-1}\rfloor}(\sn)\big|\geq Bn\Big) \\[0.7em] & =\P\Big(\bm{L}^{(n)}_{\lfloor an^{\kappa-1}\rfloor}(\sn-\rn)\geq Bn\Big)\P\Big(\bm{L}^{(n)}_{\lfloor an^{\kappa-1}\rfloor}(\tn-\sn)\big|\geq Bn\Big) \\[0.7em] & \leq\frac{4}{B^2}\frac{(\sn-\rn)(\tn-\sn)}{n^2}\leq\frac{4}{B^2}\frac{(\tn-\rn)^2}{n^2}\leq\frac{16}{B^2}(t-r)^2,
\end{align*}
where we have used the Markov inequality and the fact that $\E[L^{(p)}_k]=2p$ for the first inequality. We then obtain the convergence in law under $\P$ \eqref{ConvAnnealedProcess} by using Theorem 13.5 in \cite{BillingsleyBis}.

\vspace{0.2cm}

\noindent \textbf{Step 3}: convergence of the reduced local time under $\P^{\mathcal{E}}$

\vspace{0.1cm}

\noindent We show in this step that the convergence \eqref{ConvAnnealedProcess} actually holds under the quenched probability $\P^{\mathcal{E}}$ for $\Pb^*$-almost every environment $\mathcal{E}$. Precisely, using the same argument as in proof of Theorem 1.1 in \cite{AidRap}, one can show that for any continuous and bounded function $\Phi:D([0,\infty),\R)\to\R$, $\Pb^*$-almost surely
\begin{align*}
    \E^{\mathcal{E}}\Big[\Phi\Big(\big(\bm{L}^{(n)}_{\lfloor an^{\kappa-1}\rfloor}(\tn)/n;\; t\in[0,M]\big)\Big)\Big]-\E\Big[\Phi\Big(\big(\bm{L}^{(n)}_{\lfloor an^{\kappa-1}\rfloor}(\tn)/n;\; t\in[0,M]\big)\Big)\Big]\underset{n\to\infty}{\longrightarrow}0,
\end{align*}
for any $M>0$. Besides, we know thanks to \textbf{Step 2} that $$\lim_{n\to\infty}\E[\Phi((\bm{L}^{(n)}_{\lfloor an^{\kappa-1}\rfloor}(\tn)/n;\; t\in[0,M]))]=\E[\Phi((2t\mathcal{Y}^{(\kappa)}_{a/t^{\kappa-1}};\; t\in[0,M]))],$$ and in particular $(\E^{\mathcal{E}}[\Phi((\bm{L}^{(n)}_{\lfloor an^{\kappa-1}\rfloor}(\tn)/n;\; t\in[0,M]))])_{n\geq 1}$ converges to the same limit for $\Pb^*$-almost every environment. \\
Let us recall how to prove this. Using results of section \ref{SectionReducedProcesses}, one can see, at least along a sub-sequence, that $\lim_{n\to\infty}\P^{\mathcal{E}}( \forall\; |y|\geq\lfloor(\log n)^3\rfloor,\exists\; x\in\Bar{\mathcal{B}}^{(n)}:\; x\leq y,\; \Bar{\mathcal{B}}^{(n)}\subset\mathcal{B}^{(n)},\;\CardRoots\leq n^{3/2}\leq B^{(n)}\leq n^3,\; \forall x\in\mathcal{B}^{(n)}, \lfloor(\log n)^2\rfloor\leq|x|\leq\lfloor(\log n)^3\rfloor)=1$, $\Pb^*$-almost surely, so we restrict ourselves to this event. Note that we can always find a non-random number $\mathfrak{z}>0$ such that $\min_{|u|=\lfloor(\log n)^2\rfloor}V(u)>\mathfrak{z}(\log n)^2$ $\Pb^*$-almost surely. Indeed, since $\psi'(1)<0$ and $\psi(1)=0$, there exists $t_0>0$ such that $\frac{\psi(t_0)}{t_0}<0$. Fix $\mathfrak{z}=-\frac{\psi(t_0)}{2t_0}>0$. We have 
\begin{align*}
    \Pb\Big(\min_{|u|=\lfloor(\log n)^2\rfloor}V(u)>\mathfrak{z}(\log n)^2\Big)\leq\Eb\Big[\sum_{|u|=\lfloor(\log n)^2\rfloor}\un_{\{V(u)>\mathfrak{z}(\log n)^2\}}\Big]\leq e^{\lfloor(\log n)^2\rfloor\frac{\psi(t_0)}{2t_0}},
\end{align*}
and $\sum_{n\geq 1}\Pb(\min_{|u|=\lfloor(\log n)^2\rfloor}V(u)>\mathfrak{z}(\log n)^2)<\infty$. The Borel-Cantelli Lemma yields our claim. We also restrict ourselves to those vertices. However, in order to simplify the notations, we no longer refer to these restrictions in this step. Let $$\mathbfcal{U}_n:=\{(U_1,U_2);\;U_1\textrm{ and }U_2\textrm{ subsets of }\mathcal{U}_n\textrm{ with cardinal smaller than }n^{3}\},$$ where $\mathcal{U}_n:=\bigcup_{j\geq\lfloor(\log n)^2\rfloor}\N^j$, and define $\mathcal{A}_n$ to be the subset of $\mathbfcal{U}_n$ such that $(U_1,U_2)\in\mathcal{A}_n$ if and only if for any $(x,y)\in(U_1,U_2)$, neither $x$ is an ancestor of $y$, nor $y$ is an ancestor of $x$. Let $$\mathbfcal{X}^{(n)}_a:=(\bm{L}^{(n)}_{\lfloor an^{\kappa-1}\rfloor}(\tn)/n;\; t\in[0,M]).$$ The idea is to prove that $\E^{\mathcal{E}}[\Psi(\mathbfcal{X}^{(n)}_a)]$ concentrates around its mean. For that, one can see that
\begin{align*}
    \Eb\Big[\E^{\mathcal{E}}\big[\Psi(\mathbfcal{X}^{(n)}_a)\big]^2\Big]=\Big(\sum_{(U_1,U_2)\in\mathcal{A}_n}+\sum_{(U_1,U_2)\in\mathbfcal{U}_n\setminus\mathcal{A}_n}\Big)\Eb\Big[\E^{\mathcal{E}}\big[\Psi(\mathbfcal{X}^{(n)}_a)\un_{\{\mathcal{B}^{(n)}=U_1\}}\big]\E^{\mathcal{E}}\big[\Psi(\mathbfcal{X}^{(n)}_a)\un_{\{\mathcal{B}^{(n)}=U_2\}}\big]\Big].
\end{align*}
By definition of the reduced range
\begin{align*}
    &\sum_{(U_1,U_2)\in\mathcal{A}_n}\Eb\Big[\E^{\mathcal{E}}\big[\Psi(\mathbfcal{X}^{(n)}_a)\un_{\{\mathcal{B}^{(n)}=U_1\}}\big]\E^{\mathcal{E}}\big[\Psi(\mathbfcal{X}^{(n)}_a)\un_{\{\mathcal{B}^{(n)}=U_2\}}\big]\Big] \\ & =\sum_{(U_1,U_2)\in\mathcal{A}_n}\Eb\Big[\P^{\mathcal{E}}\big(\mathcal{B}^{(n)}=U_1\big)\P^{\mathcal{E}}\big(\mathcal{B}^{(n)}=U_2\big)\E^{\mathcal{E}}\big[\Psi(\mathbfcal{X}^{(n)}_a)\big|\mathcal{B}^{(n)}=U_1\big]\E^{\mathcal{E}}\big[\Psi(\mathbfcal{X}^{(n)}_a)\big|\mathcal{B}^{(n)}=U_2\big]\Big] \\ & = \sum_{(U_1,U_2)\in\mathcal{A}_n}\Eb\Big[\P^{\mathcal{E}}\big(\mathcal{B}^{(n)}=U_1\big)\P^{\mathcal{E}}\big(\mathcal{B}^{(n)}=U_2\big)\Big]\E\big[\Psi(\mathbfcal{X}^{(n)}_a)\big]^2\leq\E\big[\Psi(\mathbfcal{X}^{(n)}_a)\big]^2.
\end{align*}
For the second sum, we have
\begin{align*}
    &\sum_{(U_1,U_2)\in\mathbfcal{U}_n\setminus\mathcal{A}_n}\Eb\Big[\E^{\mathcal{E}}\big[\Psi(\mathbfcal{X}^{(n)}_a)\un_{\{\mathcal{B}^{(n)}=U_1\}}\big]\E^{\mathcal{E}}\big[\Psi(\mathbfcal{X}^{(n)}_a)\un_{\{\mathcal{B}^{(n)}=U_2\}}\big]\Big] \\ & \leq\|\Psi\|_{\infty}^2\sum_{(U_1,U_2)\in\mathbfcal{U}_n\setminus\mathcal{A}_n}\Eb\Big[\P^{\mathcal{E}}\big(\mathcal{B}^{(n)}=U_1\big)\P^{\mathcal{E}}\big(\mathcal{B}^{(n)}=U_2\big)\Big] \\ & \leq 2\|\Psi\|_{\infty}^2\sum_{(U_1,U_2)\in\mathbfcal{U}_n}\un_{\{\exists\;(x,y)\in U_1\times U_2:\; y\leq x\}}\Eb\Big[\P^{\mathcal{E}}\big(\mathcal{B}^{(n)}=U_1\big)\P^{\mathcal{E}}\big(\mathcal{B}^{(n)}=U_2\big)\Big] \\ & \leq 2\|\Psi\|_{\infty}^2\sum_{(U_1,U_2)\in\mathbfcal{U}_n}\;\sum_{x\in U_1}\Eb\Big[\P^{\mathcal{E}}\big(\mathcal{B}^{(n)}=U_1\big)\P^{\mathcal{E}}\big(\mathcal{B}^{(n)}=U_2,\;N_{x_n}^{(n^2)}\geq 1\big)\Big],
\end{align*}
where $x_n$ stands for the ancestor of $x$ in generation $\lfloor(\log n)^2\rfloor$ and we have used that $x_n\leq y$ and $N_y^{(n^2)}\geq 1$ which implies that $N_{x_n}^{(n^2)}\geq 1$. Note that $$\P^{\mathcal{E}}\big(N_{x_n}^{(n^2)}\geq 1\big)\leq n^2e^{-V(x_n)}\leq n^2e^{-\min_{|u|=\lfloor(\log n)^2\rfloor}V(u)}\leq n^2e^{-\mathfrak{z}\lfloor(\log n)^2\rfloor},$$ where we have used the Markov inequality and the fact that $\E^{\mathcal{E}}[N_{x_n}^{(n^2)}]=e^{-V(x_n)}$ (see below \eqref{ProbaAlpha}) for the first inequality, thus giving
\begin{align*}
    \sum_{(U_1,U_2)\in\mathbfcal{U}_n\setminus\mathcal{A}_n}\Eb\Big[\E^{\mathcal{E}}\big[\Psi(\mathbfcal{X}^{(n)}_a)\un_{\{\mathcal{B}^{(n)}=U_1\}}\big]\E^{\mathcal{E}}\big[\Psi(\mathbfcal{X}^{(n)}_a)\un_{\{\mathcal{B}^{(n)}=U_2\}}\big]\Big]\leq 2\|\Psi\|_{\infty}^2n^{5}e^{-\mathfrak{z}\lfloor(\log n)^2\rfloor}.
\end{align*}
Therefore, $\sum_{n\geq 1}\Pb(|\E^{\mathcal{E}}[\Psi(\mathbfcal{X}^{(n)}_a)]-\E[\Psi(\mathbfcal{X}^{(n)}_a)]|>\varepsilon)<\infty$ for any $\varepsilon>0$ and this concludes \textbf{Step 3}.

\vspace{0.2cm}

\noindent \textbf{Step 4}: convergence in $\Pb^*$-law of the local time

\vspace{0.1cm}

\noindent We have reached the final step of the proof and we are now ready to show \eqref{ConvPlaw}, that is, in $\Pb^*$-law
\begin{align*}
    \frac{1}{n}\bm{L}^{(n)}_{\lfloor an^{\kappa-1}\rfloor}\big(\Bar{B}^{(n)}\big)\underset{n\to\infty}{\longrightarrow} 2W_{\infty}\mathcal{Y}^{(\kappa)}_{a/(W_{\infty})^{\kappa-1}}\;.
\end{align*}
Let $\lambda\geq 0$. By \textbf{Step 3}, we have, for any $M>0$, $\Pb^*$-almost surely
\begin{align*}
    \E^{\mathcal{E}}\Big[\sup_{t\in[0,M]}e^{-\frac{\lambda}{n}\bm{L}^{(n)}_{\lfloor an^{\kappa-1}\rfloor}(\tn)}\Big]\underset{n\to\infty}{\longrightarrow}\E^{\mathcal{E}}\Big[\sup_{t\in[0,M]}e^{-\lambda 2t\mathcal{Y}^{(\kappa)}_{a/t^{\kappa-1}}}\Big],
\end{align*}
and we conclude by recalling that $\Bar{B}^{(n)}/n\to W_{\infty}$ in $\P^*$-probability, see \eqref{ConvCardRoot}, and this ends the proof of our theorem. \\

\noindent We end the article with the proof of our corollary.

\subsection{Proof of Corollary \ref{Corollary1}}\label{ProofCorollary}

Again, we only deal with the case $\kappa\in(1,2)$, the proof is the same when $\kappa\geq 2$. Corollary \ref{Corollary1} is a direct consequence of Theorem \ref{Theorem1} and Theorem \ref{Theorem3}. Indeed, on the one hand, in $\Pb^*$-probability
\begin{align*}
    \E^{\mathcal{E}}\Big[e^{-\frac{\lambda}{n}L^{(1)}_{\lfloor n^{\kappa-1}\rfloor}}\Big]^n=\E^{\mathcal{E}}\Big[e^{-\frac{\lambda}{n}L^{(n)}_{\lfloor n^{\kappa-1}\rfloor}}\Big]\underset{n\to\infty}{\longrightarrow}\E^{\mathcal{E}}\Big[e^{-2\lambda W_{\infty}\mathcal{Y}^{(\kappa)}_{1/(W_{\infty})^{\kappa-1}}}\Big]=e^{-\bm{\mathrm{C}}_{\kappa,1}W_{\infty}(1-\phi^{(\kappa)}(2\lambda/\bm{\mathrm{C}}_{\kappa,1}))},
\end{align*}
where we have used Theorem \ref{Theorem3} for this convergence and we recall that $\bm{\mathrm{C}}_{\kappa,b}=(b(\kappa-1)C_{\infty}\bm{\mathrm{C}}_{\kappa})^{-1/(\kappa-1)}$, see below equation \eqref{JointCVinLaw}. On the other hand
\begin{align*}
    \E^{\mathcal{E}}\Big[e^{-\frac{\lambda}{n}L^{(1)}_{\lfloor n^{\kappa-1}\rfloor}}\Big]=1-\P^{\mathcal{E}}\big(L^{(1)}_{\lfloor n^{\kappa-1}\rfloor}>0\big)\Big(1-\E^{\mathcal{E}}\Big[e^{-\frac{\lambda}{n}L^{(1)}_{\lfloor n^{\kappa-1}\rfloor}}\Big|L^{(1)}_{\lfloor n^{\kappa-1}\rfloor}>0\Big]\Big),
\end{align*}
so Theorem \ref{Theorem1} immediately gives, in $\Pb^*$-probability
\begin{align*}
    \lim_{n\to\infty}\E^{\mathcal{E}}\Big[e^{-\frac{\lambda}{n}L^{(1)}_{\lfloor n^{\kappa-1}\rfloor}}\Big|L^{(1)}_{\lfloor n^{\kappa-1}\rfloor}>0\Big]=\phi^{(\kappa)}\big(2\lambda/\bm{\mathrm{C}}_{\kappa,1}\big),
\end{align*}
which is exactly what we wanted.

\vspace{0.2cm}

\section{Table of important notation}

\subsection{Important notation related to the random walk $\X$}

$\tau^j=\inf\{k>\tau^{j-1};\; X_{k-1}=e^*,\; X_k=e\}$ for $j\geq 1$ and $\tau^0=0$, section \ref{SectionRange}.

\vspace{0.1cm}

\noindent $\VectCoord{N}{p}_x=\sum_{j=1}^{\tau^p}\un_{\{X_{j-1}=x^*,\; X_j=x\}}$, the edge local time of $(x^*,x)$ at time $\tau^p$ with $p\geq 1$ and $x\not=e^*$, section \ref{SectionRange}.

\vspace{0.1cm}

\noindent$\VectCoord{\mathcal{R}}{p}=\big\{x\in\T;\; \VectCoord{N}{p}_x\geq 1\big\}$, the range of $\X$ at time $\tau^p$, section \ref{SectionRange}.

\vspace{0.1cm}

\noindent$\VectCoord{L}{p}_k=\sum_{j=1}^{\tau^p}\un_{\{|X_j|=k\}}$, the local time of $\X$ at level $k$ and at time $\tau^p$, equation \eqref{DefLocalTimes}.

\vspace{0.1cm}

\noindent $Z^{(p)}_k=\sum_{|x|=k}N_x^{(p)}$, the  multi-type additive martingale, equation \eqref{DefMultiMart}.

\vspace{0.1cm}

\noindent $p^{\mathcal{E}}_k=\P^{\mathcal{E}}(Z^{(1)}_k>0)$ and $p_k=\Eb[p^{\mathcal{E}}_k]=\P(Z^{(1)}_k>0)$, right after equation \eqref{PhiKappa}.

\vspace{0.1cm}

\noindent $G^{\mathcal{E}}_{k}(s,p)=\E^{\mathcal{E}}[s^{Z^{(p)}_k}]$ and $G_{k}(s,p)=\Eb[G_k(s,p)]=\E[s^{Z^{(p)}_k}]$, beginning of section \ref{SectionPR}.

\vspace{0.1cm}

\noindent $G^{\mathcal{E}}_{k}(s)=G^{\mathcal{E}}_{k}(s,1)$ and $G_{k}(s)=G_{k}(s,1)$, beginning of section \ref{SectionPR}.

\subsection{Important notation related to the reduced range}

$\mathcal{B}^{(p)}_{\ell}=\big\{x\in\T,\; |x|>\ell;\; N_x^{(p)}=1\textrm{ and }\min_{\ell<i<|x|}N_{x_i}^{(p)}\geq 2\big\}$, the set of roots of the reduced forest, equation \eqref{SetRoots}.

\vspace{0.1cm}

\noindent$\mathcal{B}^{(n)}=\mathcal{B}^{(n^2)}_{\lfloor(\log n)^2\rfloor}$ and $\Bar{\mathcal{B}}^{(n)}=\mathcal{B}^{(n)}_{\lfloor(\log n)^2\rfloor}$, right after equation \eqref{ConvCardRoot}.

\vspace{0.1cm}

\noindent $B^{(n)}$ is the cardinal of $\mathcal{B}^{(n)}$ and $\Bar{B}^{(n)}$ the cardinal of $\Bar{\mathcal{B}}^{(n)}$, right after equation \eqref{ConvCardRoot}.

\vspace{0.1cm}

\noindent $\bm{t}_x^{(n)}=N_x^{(n^2)}$, after equation \eqref{ReducedRange}.

\vspace{0.1cm}

\noindent$\bm{Z}^{(n)}_{k}(q)=\sum_{i=1}^{q}\bm{Z}^{(n,i)}_k$, the reduced multi-type additive martingale with $\bm{Z}^{(n,i)}_k=\sum_{x\in\mathbfcal{R}^{(n)}_i}\un_{\{|x|_i=k\}}\bm{t}^{(n)}_x$, equation \eqref{RedAddMart}.

\vspace{0.1cm}

\noindent$\bm{L}^{(n)}_{k}(q)=\sum_{i=1}^{q}\bm{L}^{(n,i)}_k$, the reduced local time with $\bm{L}^{(n,i)}_{k}=\bm{Z}^{(n,i)}_{k}+\bm{Z}^{(n,i)}_{k+1}$, equation \eqref{RedLocalTime}.

\vspace{0.5cm}

\noindent\begin{merci}
    The author is partially supported for this work by the NZ Royal Society Te Apārangi Marsden Fund (22-UOA-052) entitled "Genealogies of samples of individuals selected at random from stochastic populations: probabilistic structure and applications". The author also benefited during the preparation of this paper from the support of Hua Loo-Keng Center for Mathematical Sciences (AMSS, Chinese Academy of Sciences) and the National Natural Science Foundation of China (No. 12288201). \\
    Finally, I would like to express my sincere thanks to an anonymous referee for her/his very careful reading of the paper, her/his relevant and precise remarks that were very useful to help improve the paper.
\end{merci}

\bibliographystyle{alpha}
\bibliography{thbiblio}

@Article{Hu2017,
author={Y. Hu},
title={Local Times of Subdiffusive Biased Walks on Trees},
journal= {J. of Theoret. Probab.},
year={2017},
volume={30},
number={2},
pages={529--550},
}

@article{AndChen,
    AUTHOR = {Andreoletti, Pierre and Chen, Xinxin},
     TITLE = {Range and critical generations of a random walk on
              {G}alton-{W}atson trees},
   JOURNAL = {Ann. Inst. Henri Poincar\'e Probab. Stat.},
  FJOURNAL = {Annales de l'Institut Henri Poincar\'e Probabilit\'es et
              Statistiques},
    VOLUME = {54},
      YEAR = {2018},
    NUMBER = {1},
     PAGES = {466--513},
      ISSN = {0246-0203},
   MRCLASS = {60K37 (05C81 60G50 60J80)},
  MRNUMBER = {3765897},
       DOI = {10.1214/16-AIHP812},
       URL = {https://doi.org/10.1214/16-AIHP812},
}

@article{AndDeb1,
	Author = {P. Andreoletti and P. Debs},
	Journal = {J. Theoret. Probab.},
	Pages = {\ 518 -- 538},
	Title = {The number of generations entirely visited for recurrent random walks on random environment},
	Volume = {27},
	Year = {2014}}

@article{AndDeb2,
	Author = {P. Andreoletti and P. Debs},
	Journal = {Electronic Journal of Probab.},
	Pages = {1--22},
	Title = {Spread of visited sites of a random walk along the generations of a branching process},
	Volume = {19},
	Year = {2014}}

@article{AndDiel3,
title = "The heavy range of randomly biased walks on trees",
journal = "Stochastic Processes and their Applications",
volume = "130",
number = "2",
pages = "962 - 999",
year = "2020",
issn = "0304-4149",
doi = "https://doi.org/10.1016/j.spa.2019.04.004",
author = "Pierre Andreoletti and Roland Diel",
}

@Article{Aidekon2008,
author={A{\"i}d{\'e}kon, E.},
title={Transient random walks in random environment on a {G}alton-{W}atson tree},
journal={Probability Theory and Related },
year={2008},
OPTmonth={Nov},
OPTday="01",
volume={142},
number={3},
pages={525--559},
}

@Article{AidRap,
author="E. A{\"i}d{\'e}kon 
and  L. de Raph{\'e}lis",
title="Scaling limit of the recurrent biased random walk on a {G}alton-{W}atson tree",
journal="Probability Theory and Related Fields",
year="2017",
day="01",
volume="169",
number="3",
pages="643--666",
issn="1432-2064",
doi="10.1007/s00440-016-0739-8",
url="https://doi.org/10.1007/s00440-016-0739-8"
}

@article{HuShi15b,
	Author = {Y. Hu and Z. Shi},
	Journal = {arXiv : http://arxiv.org/abs/1403.6799},
	Optyear = {2000},
	Title = {The potential energy of biased random walks on trees},
	Year = {2016}}

@article{HuShi10a,
	Author = {Y. Hu and Z. Shi},
	Journal = {Ann. Probab.},
	Pages = {1978-1997},
	Title = {Slow movement of recurrent random walk in random environment on a regular tree},
	Volume = {{35}},
	Year = {2007}}

@article{HuShi10b,
	Author = {G. Faraud and Y. Hu and Z. Shi},
	Journal = {Probab. Theory Relat. Fields},
	pages = {621-660},
	VOLUME = {154},
	Title = {Almost sure convergence for stochastically biased random walks on trees},
	Year = {2011}}

@article{Faraud,
	Author = {G. Faraud},
	Journal = {Electronic Journal of Probability},
	Number = {6},
	Pages = {174-215},
	Title = {A central limit theorem for random walk in a random environment on marked {G}alton-{W}atson trees},
	Volume = {{16}},
	Year = {2011}}

@article{Lyons,
	Author = {R. Lyons},
	Journal = {Ann. Probab.},
	Pages = {931--958},
	Title = {Random Walks and Percolation on Trees},
	Volume = {18},
	Year = {1990}}

@article{Lyons2,
	Author = {R. Lyons},
	Journal = {Ann. Probab.},
	Pages = {2043--2088},
	Title = {Random Walks, capacity and Percolation on Trees},
	Volume = {20},
	Year = {1992}}

@article{LyonPema,
	Author = {Russell Lyons and Robin Pemantle},
	Journal = {Annals of Probability},
	Pages = {125--136},
	Title = {Random Walk in a Random Environment and First-Passage Percolation on Trees},
	Volume = {20},
	Year = {1992}}

@article{AndAKHightPotential,
author = {Pierre Andreoletti and Alexis Kagan},
title = {{Generalized range of slow random walks on trees}},
volume = {60},
journal = {Annales de l'Institut Henri Poincaré, Probabilités et Statistiques},
number = {2},
publisher = {Institut Henri Poincaré},
pages = {1458 -- 1509},
year = {2024},
doi = {10.1214/23-AIHP1367},
URL = {https://doi.org/10.1214/23-AIHP1367}
}

@article{LyonsRussellPemantle1,
author = {Lyons, Russell and Pemantle, Robin and Peres, Yuval},
year = {1996},
month = {10},
pages = {},
title = {Ergodic Theory on {G}alton-{W}atson Trees: Speed of Random Walk and Dimension of Harmonic Measure},
volume = {15},
journal = {Ergodic Theory and Dynamical Systems},
doi = {10.1017/S0143385700008543}
}

@article{LyonsRussellPemantle2,
author = {Lyons, Russell and Pemantle, Robin and Peres, Yuval},
year = {1996},
month = {10},
pages = {},
title = {Biased Random Walks on {G}alton-{W}atson Trees},
volume = {106},
journal = {Probability Theory and Related Fields},
doi = {10.1007/s004400050064}
}

@Inbook{Lyons1997,
author="Lyons, Russell",
editor="Athreya, Krishna B.
and Jagers, Peter",
title="A Simple Path to Biggins' Martingale Convergence for Branching Random Walk",
bookTitle="Classical and Modern Branching Processes",
year="1997",
publisher="Springer New York",
address="New York, NY",
pages="217--221",
isbn="978-1-4612-1862-3",
doi="10.1007/978-1-4612-1862-3_17",
url="https://doi.org/10.1007/978-1-4612-1862-3_17"
}

@article{Liu1,
title = {On generalized multiplicative cascades},
journal = {Stochastic Processes and their Applications},
volume = {86},
number = {2},
pages = {263-286},
year = {2000},
issn = {0304-4149},
doi = {https://doi.org/10.1016/S0304-4149(99)00097-6},
url = {https://www.sciencedirect.com/science/article/pii/S0304414999000976},
author = {Quansheng Liu},
}

@article{AidekonSpeed,
author = {A{\"i}d{\'e}kon, Elie},
year = {2014},
month = {},
pages = {597--617},
title = {Speed of the biased random walk on a {G}alton–{W}atson tree},
volume = {159},
number = {3},
journal = {Probability Theory and Related Fields},
doi = {10.1007/s00440-013-0515-y}
}

@misc{ChenDeRaphMax,
doi = {10.48550/ARXIV.2009.13816},
url = {https://arxiv.org/abs/2009.13816},
author = {Chen, Xinxin and de Raphélis, Loïc},
title = {Maximal local time of randomly biased random walks on a {G}alton-{W}atson tree},
publisher = {arXiv},
year = {2020},
copyright = {arXiv.org perpetual, non-exclusive license}
}

@article{deRaph1,
author = {Loïc de Raphélis},
title = {{Scaling limit of the subdiffusive random walk on a Galton–Watson tree in random environment}},
volume = {50},
journal = {The Annals of Probability},
number = {1},
publisher = {Institute of Mathematical Statistics},
pages = {339 -- 396},
year = {2022},
doi = {10.1214/21-AOP1535},
URL = {https://doi.org/10.1214/21-AOP1535}
}

@article{DuqLeGall,
author = {Duquesne, Thomas and Le Gall, Jean-Francois},
year = {2005},
month = {10},
pages = {},
title = {Random Trees, Levy Processes and Spatial Branching Processes},
volume = {281},
journal = {Astérisque}
}

@book{BillingsleyBis,
Author = {P. Billingsley},
Publisher = {John Wiley Sons},
Title = {Convergence of Probability Measures, Second Edition},
Year = {1999}
}

@article{BA_F_G_H,
author = {G{\'e}rard Ben Arous and Alexander Fribergh and Nina Gantert and Alan Hammond},
title = {{Biased random walks on Galton–Watson trees with leaves}},
volume = {40},
journal = {The Annals of Probability},
number = {1},
publisher = {Institute of Mathematical Statistics},
pages = {280 -- 338},
year = {2012},
doi = {10.1214/10-AOP620},
URL = {https://doi.org/10.1214/10-AOP620}
}

@article{AK23LocalTimes,
Author = {Alexis Kagan},
Journal = {Preprint},
Title = {Scaling limit for local times and return times of a randomly biased walk on a {G}alton-{W}atson tree},
Volume = {},
Year = {2023}
}

@article{RousselinConduc,
Author = {Pierre Rousselin},
journal={ ALEA-Latin American Journal of Probability and Mathematical Statistics },
title={ Conductance of a subdiffusive random weighted tree },
Volume = { 20 },
pages={ 885-903 },
number={ 1 },
Year = {2023}
}

@article{Biggins1977, 
title={Martingale convergence in the branching random walk}, 
volume={14}, 
DOI={10.2307/3213258}, 
number={1}, 
journal={Journal of Applied Probability}, author={Biggins, J. D.}, year={1977}, pages={25–37}}

@article{Alsmeyer_Iksanov,
author = {Gerold Alsmeyer and Alex Iksanov},
title = {{A Log-Type Moment Result for Perpetuities and Its Application to Martingales in Supercritical Branching Random Walks}},
volume = {14},
journal = {Electronic Journal of Probability},
number = {none},
publisher = {Institute of Mathematical Statistics and Bernoulli Society},
pages = {289 -- 313},
year = {2009},
doi = {10.1214/EJP.v14-596},
URL = {https://doi.org/10.1214/EJP.v14-596}
}

@article{Slack1968,
  title={A branching process with mean one and possibly infinite variance},
  author={R. S. Slack},
  journal={Zeitschrift f{\"u}r Wahrscheinlichkeitstheorie und Verwandte Gebiete},
  year={1968},
  volume={9},
  pages={139-145},
  url={https://api.semanticscholar.org/CorpusID:122681681}
}

@inproceedings{yaglom1947,
  title={Some limit theorems of theory of Branching random processes},
  author={Yaglom, AM},
  booktitle={Doklady AN SSSR},
  volume={13},
  pages={795--798},
  year={1947}
}

@article{KestenNeySpitzer,
author = {Kesten, H. and Ney, P. and Spitzer, F.},
title = {The Galton-Watson Process with Mean One and Finite Variance},
journal = {Theory of Probability \& Its Applications},
volume = {11},
number = {4},
pages = {513-540},
year = {1966},
doi = {10.1137/1111059}
}

@article{Slack1972,
  title={Further notes on branching processes with mean 1},
  author={Slack, RS},
  journal={Zeitschrift f{\"u}r Wahrscheinlichkeitstheorie und Verwandte Gebiete},
  volume={25},
  number={1},
  pages={31--38},
  year={1972},
  publisher={Springer}
}

@book{LivreBertoin,
    author = {Bertoin, Jean},
    title = {Lévy processes},
    publisher = {Cambridge University Press},
    series = {Cambridge Tracts in Mathematics},
    volume = {121},
    year = {1996}
}

\end{document}